\documentclass[a4paper,10pt,reqno]{amsart}
\usepackage{graphicx}
\usepackage{mathrsfs}
\usepackage{amssymb}
\usepackage[colorlinks=true]{hyperref}
\usepackage{enumitem}
\hypersetup{urlcolor=blue, citecolor=red}
\newtheorem{theorem}{Theorem}[section]
\newtheorem{corollary}[theorem]{Corollary}
\newtheorem{lemma}[theorem]{Lemma}
\newtheorem{proposition}[theorem]{Proposition}
\newtheorem{remark}[theorem]{Remark}

\theoremstyle{definition}

\numberwithin{equation}{section}

\newcommand{\Lr}{L^1(\Omega){\times}L^1(\Omega)}

\newcommand{\Sn}{\mathbb{S}^{n-1}}

\newcommand{\Hr}{H^1(\Omega){\times}H^1(\Omega)}

\newcommand{\qo}{\mathcal{L}^n\text{-a.e.\ in } \Omega}

\newcommand{\ve}{\varepsilon}
\newcommand{\eps}{\varepsilon}

\newcommand{\R}{{\mathbb R}}

\newcommand{\N}{{\mathbb N}}
\newcommand{\Rn}{{\R}^n}
\newcommand{\Hn}{{\mathcal H}^{n-1}}
\newcommand{\Ho}{{\mathcal H}^0}
\newcommand{\Ln}{{\mathcal L}^n}
\newcommand{\Le}{{\mathcal L}^1}
\newcommand{\ba}[1]{\begin{eqnarray} #1 \end{eqnarray}}
\newcommand{\be}[1]{\begin{equation} #1 \end{equation}}
\newcommand{\bm}[1]{\begin{multline} #1 \end{multline}}
\newcommand{\bes}[1]{\begin{eqnarray*} #1 \end{eqnarray*}}

\newcommand{\li}{\liminf_{k\to+\infty}}
\newcommand{\ls}{\limsup_{k\to+\infty}}
\newcommand{\aj}{a}
\newcommand{\f}[1]{f^{(#1)}}
\newcommand{\fk}[2]{f^{(#1)}_{#2}}
\newcommand{\F}[1]{F^{(#1)}}
\newcommand{\Fk}[2]{F^{(#1)}_{#2}}
\newcommand{\D}{\mathscr{D}}
\newcommand{\MSt}{\widetilde{{M\!S}}}
\newcommand{\MS}{{M\!S}}
\newcommand{\Dt}{\widetilde{\D}}
\newcommand{\gD}[1]{1\wedge\ell #1}
\DeclareMathOperator*{\diam}{diam}
\DeclareMathOperator*{\esssup}{ess\,sup}
\DeclareMathOperator*{\essinf}{ess\,inf}
\newcounter{rem}[theorem]
\renewcommand*{\therem}{\thetheorem.\arabic{rem}}

	\newcommand{\dist}{\mathrm{dist}}
\newcommand\res{\mathop{\hbox{\vrule height 7pt width .3pt depth 0pt
\vrule height .3pt width 5pt depth 0pt}}\nolimits}

\mathchardef\emptyset="001F
\title[Phase field approximation of cohesive fracture models]{Phase field approximation of cohesive fracture models}


\address{Universit\"{a}t Bonn}
\email{sergio.conti@uni-bonn.de}
\email{iurlano@iam.uni-bonn.de}
\address{Universit\`a di Firenze}
\email{focardi@math.unifi.it}

\author[S. Conti, M. Focardi, and F. Iurlano]{S. Conti -- M. Focardi -- F. Iurlano}

\address{IAM, Universit\"{a}t Bonn\\
Endenicher Allee 60, 53115 -- Bonn, Germany}

\address{IAM and HCM, Universit\"{a}t Bonn\\
Endenicher Allee 60, 53115 -- Bonn, Germany}

\address{DiMaI ``U. Dini'', Universit\`a degli Studi di Firenze\\ 
Viale Morgagni 67/A, 50134 -- Firenze, Italy}

\begin{document}
\begin{abstract}
We obtain a cohesive fracture model as a $\Gamma$-limit of scalar damage models in
which the elastic coefficient is computed from the damage variable $v$ through
a function $f_k$ of the form $f_k(v)=\mathrm{min}\{1,\eps_k^{1/2} f(v)\}$, with $f$
diverging for $v$ close to the  value describing undamaged material. The resulting 
fracture energy  can be determined by solving a
one-dimensional vectorial optimal profile problem. It 
is linear in the opening $s$ at small values of $s$ and
has a finite limit as $s\to\infty$.
If the function $f$ is allowed to  depend on the index $k$, 
for specific choices we recover in the limit  Dugdale's and Griffith's 
fracture models, and
 models with surface energy density having a power-law growth at small openings.
\end{abstract}
\maketitle
\section{Introduction}\label{introd}

The modeling of fracture in materials leads naturally to functional spaces with
discontinuities, in particular functions of bounded variation ($BV$) and of
bounded deformation ($BD$). In variational models, the key ingredients are a
volume term, corresponding to the stored energy and depending on the
diffuse part of the deformation gradient, and a surface term, modeling the
fracture energy and depending on the jump part of the deformation
gradient \cite{fra-mar,BFM,barenblatt,dug}. 
For antiplane shear models one can consider a scalar displacement 
$u\in BV(\Omega)$, typical models take the form
\begin{equation}\label{eqfrattura}
  \int_\Omega h(|\nabla u|) dx +   \kappa\,|D^cu|(\Omega)
  + \int_{\Omega\cap J_u} g(|[u]|) d\Hn\,.
\end{equation}
Here $h$ represents the strain energy density,  quadratic near the origin; $g$ is
the surface energy density depending on the opening $[u]$ of the crack, and $\kappa\in[0,+\infty]$ 
is a constant related to the slope of $g$ at 0 and the slope of $h$ at $\infty$. 

In models of brittle fracture one usually considers $g$ to be a constant,
given by twice the energy required to generate a free surface
\cite{fra-mar,BFM}. 
Correspondingly $\kappa=\infty$ and the Cantor part $D^cu$ disappears, so that one
can assume $u\in SBV$.
Physically this represents a situation in which already for the smallest
opening there is no interaction between the two sides of the fracture,
surface reconstruction is purely local. Analytically, the resulting functional
coincides with the Mumford-Shah functional from image segmentation.

In ductile materials fracture proceeds through the opening of a series of
voids, separated by thin filaments which produce a weak bound between the
surfaces at moderate openings \cite{barenblatt,dug,FokouaContiOrtiz2014}. The
function $g$ 
grows then continuously from 
$g(0)=0$ to some finite value $g(\infty)$, representing the energetic cost of
total fracture.  The constant $\kappa$ is its slope at 0, and $D^cu$ represents the
distribution of microcracks. For the same reason the volume energy density $h$
becomes linear at $\infty$. 

A large literature was devoted to the derivation of models like
(\ref{eqfrattura}) from more regular models, like damage or phase field
models, mainly within the framework of $\Gamma$-convergence.
These regularizations can be interpreted as microscopic physical models, 
so that the $\Gamma$-limit justifies the macroscopic model 
(\ref{eqfrattura}), or as regularizations used for example to approximate
(\ref{eqfrattura}) numerically.
Ambrosio and Tortorelli \cite{amb-tort1,amb-tort2} have shown that 
\begin{equation}\label{eqAT}
  \int_\Omega \left( (v^2+o(\eps)) |\nabla u|^2 + \frac{(1-v)^2}{4 \eps} +
  \eps|\nabla v|^2 \right) dx
\end{equation}
converges to the Mumford-Shah functional, which coincides with
(\ref{eqfrattura}) with $h(t)=t^2$, $\kappa=\infty$, $g(t)=1$.
This result was extended in many directions, for example to 
 vector-valued functions \cite{focardi,focardi_tesi}, to 
linearized elasticity \cite{chambolle,addendum,iur12}, to second-order 
problems \cite{AFM},  vectorial problems \cite{Sh},
and models with nonlinear injectivity constraints
\cite{HenaoMoracorralXu}; for numerical simulations 
 we refer to \cite{bel-cos,bou-fra-mar2,bou,bur-ort-sul,bur-ort-sul2}.
There is also a large numerical literature on the application to computer
 vision, see for example \cite{Fusco2003,BarSochenKiryati2006} and references therein. 

Models like (\ref{eqfrattura}) with a linear $h$ were obtained in \cite{AlicBrShah, AlFoc}.
The case with a quadratic $h$ and an affine $g$, i.e. $h(t)=t^2$, $g(t)=t+c$, and $\kappa=+\infty$, described
in \cite{amb-lem-roy} a strain localization plastic process. From the mathematical point of view 
the functional was limit of models like (\ref{eqAT}) with an additional term linear in $|\nabla u|$. 
In \cite{dm-iur,iur} the asymptotic behavior of a generalization of \eqref{eqAT} with different 
scalings of the three terms was analyzed. In one of the several regimes identified, the limiting model 
again exhibited an affine $g$. The result was then extended to the vectorial case in \cite{foc-iur}.
In \cite{iur} a different scaling of the parameters led to the Hencky's diffuse plasticity, i.e. to a model
like (\ref{eqfrattura}) with a linear $g$. This functional can be used to describe ductile fracture only at small openings.
Discrete models for fracture were studied for example in \cite{bra-gar-dm,BraidesGelli2006}.
Up to now, we are unaware of any result in which a ductile fracture model with $g$ 
continuous and bounded, as described above, has been derived. 

In this work we study a damage model as proposed by Pham and Marigo
\cite{PM1,PM2} (cp. Remark~\ref{r:gen}), namely,
\begin{equation}\label{eqPM}
  F_\eps(u,v):=\int_\Omega \left( f_\eps^2(v) |\nabla u|^2 + \frac{(1-v)^2}{4 \eps} + 
  \eps|\nabla v|^2 \right) dx,
\end{equation}
with $u,v\in H^1(\Omega)$, $0\leq v\leq 1$ $\qo$, and $F_\eps(u,v):=\infty$ otherwise,
and show that it converges to a cohesive fracture model like
(\ref{eqfrattura}), where $g$ is a continuous bounded function with $g(0)=0$,
which is linear close to the origin. The potential $f_\ve:[0,1)\to[0,+\infty]$ in \eqref{eqPM} is defined by
\begin{equation}\label{e:feintro}
f_\eps(s):=1\wedge \eps^{1/2}f(s),
\end{equation}
where $f\in C^0([0,1),[0,+\infty))$ is nondecreasing, $f^{-1}(0)=\{0\}$, and it 
 satisfies
\begin{equation}\label{eqdivergencefintro}
\lim_{s\to1}(1-s)f(s)=\ell, \quad \ell\in(0,+\infty).  
\end{equation}
Our main result describes the asymptotic of $(F_\eps)$ as follows.
\begin{theorem}\label{t:mainintro}
Let $\Omega\subset\R^n$ be a bounded Lipschitz set. 

Then, the functionals $F_\eps$ $\Gamma$-converge in $\Lr$ 
  to the functional $F$ defined by
\begin{equation*}
F(u,v):=\begin{cases} 
\displaystyle\int_\Omega h(|\nabla u|)dx+\int_{J_{u}}g(|[u]|)d\Hn+\ell|D^cu|(\Omega) & \textrm{if $v=1$ $\qo$, $u\in GBV(\Omega)$}\cr
+ \infty & \text{otherwise}.
\end{cases} 
\end{equation*}
Here the volume energy density $h$ is set as $h(s):=s^2$ if $s\leq\ell/2$ and as $h(s):=\ell s-\ell^2/4$ otherwise,
while the surface energy density $g$ is given by
\begin{alignat}1\nonumber
  g(s):=\inf &
\left\{
\int_0^1|1-\beta|\sqrt{f^2(\beta)|\alpha'|^2+|\beta'|^2}\,dt:\,(\alpha,\beta) 
\in H^1\big((0,1)\big),\right.
\\&\left. \phantom{\int}\hskip25mm \alpha(0)=0,\ \alpha(1)=s,\ \beta(0)=\beta(1)=1
\right\}.\label{eqintrodefg}
\end{alignat}
\end{theorem}
The key difference with the previously discussed work is that in our case the
optimal profiles for the damage variable $v$ and the elastic displacement $u$
cannot be determined separately. They instead arise from a joint vectorial
minimization problem which defines the cohesive energy $g$, specified in \eqref{eqintrodefg}. 
Figures \ref{fig1} and \ref{fig2}  show the behavior of $f_\eps$ and $g$ in the case $f(s)=s/(1-s)$, 
Figure \ref{fig3} shows the profiles $\alpha$, $\beta$ entering  (\ref{eqintrodefg}).

Theorem~\ref{t:mainintro}, in the equivalent formulation given in Theorem \ref{t:gamma-lim} below,  
is proved first in the one-dimensional case in  Section~\ref{s:onedim},
relying on elementary arguments in which we estimate separately the diffuse and jump contributions,
and then extended to the general $n$-dimensional setting in Section~\ref{s:ndim}. This  
extension is obtained by means of several tools. A slicing technique and the above mentioned 
one-dimensional result are the key for the lower bound inequality. 
Instead, the upper bound inequality is proved through the direct methods of $\Gamma$-convergence
on $SBV$, i.e. abstract compactness results and integral representation of the corresponding 
$\Gamma$-limits. The latter methods are complemented with an ad-hoc one-dimensional construction 
to match the lower bound on $SBV$ and a relaxation procedure to prove the result on $BV$. 
Finally, the extension to $GBV$ is obtained via a simple truncation argument.

The issues of equi-coercivity of $F_\eps$ and the convergence of the related minima are dealt with 
in Theorem~\ref{t:comp} and Corollary~\ref{c:min} below, respectively. 

Qualitative properties of the surface energy density $g$ defined in \eqref{eqintrodefg} are analyzed
in Section~\ref{s:gprop}. Its monotonicity, sublinearity, boundedness and linear behavior in the 
origin are established in Proposition~\ref{11}. Proposition~\ref{p2} characterizes $g$ by means of
an asymptotic cell formula particularly convenient in the proof of the $\Gamma$-limsup inequality.
Furthermore, the dependence of $g$ on $f$ is analyzed in detail in Proposition~\ref{p:gl}.
The latter results on one hand show the variety of such a class of functions, and on the other hand 
are instrumental to handle the approximation of other models. 

In Section~\ref{s:further} we discuss how the   phase field approximation scheme
can be used to approximate different fracture models. 
We first consider damage functions of the form
\begin{equation*}
  f_k(s):=\min\{1, \eps_k^{1/2} \max\{f(s), a_ks\}\}
\end{equation*}
and show that if $a_k\to\infty$ and $a_k \eps_k^{1/2}\to0$ then the a similar result holds with the limiting 
surface energy $g(s)=1\wedge (\ell s)$, so that (\ref{eqfrattura}) 
reduces to Dugdale's fracture model 
(Theorem~\ref{t:Dugdale} in Section \ref{subsecdugdale}).

Secondly we consider a situation in which $f$ diverges with exponent $p>1$ close to $s=1$, so that 
(\ref{eqdivergencefintro}) is replaced by
\begin{equation*}
\lim_{s\to1}(1-s)^p f(s)=\ell\,.
\end{equation*}
Also in this case  the functionals $\Gamma$-converge to a problem of the form 
of (\ref{eqfrattura}), in this case however the fracture energy $g$ turns out
to be proportional to the opening $s$ to the power $2/(p+1)$ at small $s$. Correspondingly the coefficient $\kappa$ of the diffuse part is infinite, so that the limiting problem is framed in the space GSBV, see Theorem  \ref{t:sublin} in
Section \ref{subsectpowerlaw}.

Finally we show that if $f_k(s)$ diverges as $\ell_k/(1-s)$, with $\ell_k\to\infty$, then  Griffith's fracture model is recovered in the limit, see Theorem \ref{t:MS} in Section \ref{subsecgriffith} below.

We finally resume the structure of the paper. In Section~\ref{s:nota} we introduce some notations,
some preliminaries, and the functional setting of the problem. 
The main result of the paper is stated in Section~\ref{s:stat}, 
where we also discuss the convergence of related minimum problems and 
minimizers.
Our $\Gamma$-convergence result relies on several properties of the 
surface energy density $g$ that are established in Section~\ref{s:gprop}.
The proof is then given first in the one-dimensional case in 
Section~\ref{s:onedim} and then in $n$ dimensions in Section \ref{s:ndim}.
The three generalizations are discussed and proven in  Section~\ref{s:further}.

\begin{figure}
\includegraphics[width=5cm]{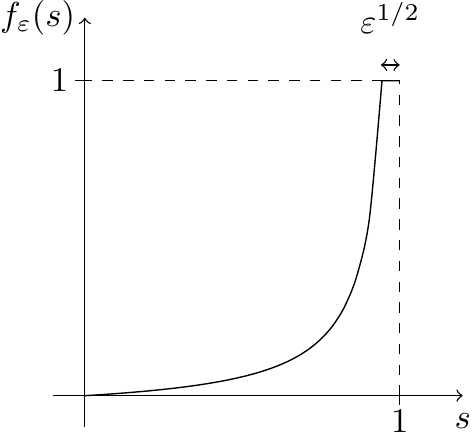}
\caption{Sketch of the function $f_\eps(s)$ for the prototypical case $f(s)=s/(1-s)$.} 
\label{fig1}
\end{figure}

\begin{figure}
\includegraphics[width=5cm]{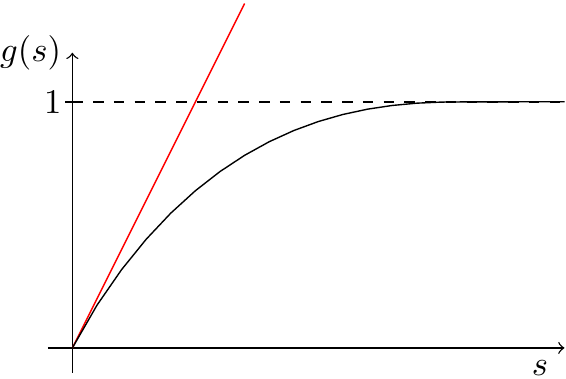}
\caption{Sketch of the function $g(s)$ defined in (\ref{eqintrodefg}), obtained by numerical 
minimization using $f(s)=s/(1-s)$ (cp. Proposition~\ref{11} and Remark~\ref{r:g}).}  
\label{fig2}
\end{figure}

\begin{figure}
\includegraphics[width=5cm]{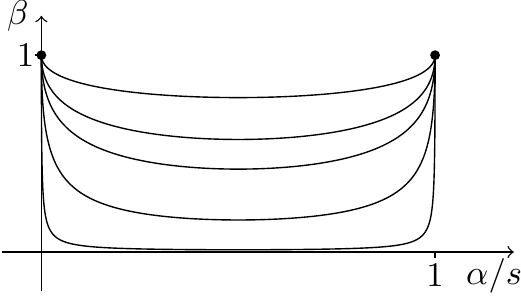}
\caption{Optimal profiles $(\alpha/s,\beta)$ obtained numerically from the minimization in the definition of 
$g(s)$, see  (\ref{eqintrodefg}), for $f(s)=s/(1-s)$ and for $s=0.1, 0.3, 0.5, 1,$ and $1.5$ (from top to bottom). 
All curves remain inside the rectangle $(0,1)\times (0,1)$ except for the two endpoints.} 
\label{fig3}
\end{figure}

\section{Notation and preliminaries}\label{s:nota}

Let $n\geq1$ be a fixed integer. We denote the Lebesgue measure and the $k$-dimensional Hausdorff measure in 
$\mathbb{R}^n$ by $\mathcal{L}^n$ and $\mathcal{H}^k$, respectively. Given $\Omega\subset\Rn$ an open bounded
set with Lipschitz boundary, we define $\mathcal{A}(\Omega)$ as the set of all open subsets of $\Omega$. 

Throughout the paper $c$ denotes a generic positive constant that can vary from line to line. 

\subsection{\texorpdfstring{\boldmath{$\Gamma$}}{Gamma}- and {\texorpdfstring{\boldmath{$\overline{\Gamma}$}}{Gamma bar}-convergence}}

Given an open set $\Omega\subset\Rn$  and a sequence of functionals 
$\mathscr{F}_k:X\times\mathcal{A}(\Omega)\to [0,+\infty]$, $(X,d)$ a separable metric space,  
such that the set function $\mathscr{F}_k(u;\cdot)$ is nondecreasing on the family $\mathcal{A}(\Omega)$ 
of open subsets of $\Omega$, set 
\[
\mathscr{F}'(\cdot; A):=\Gamma\hbox{-}\li \mathscr{F}_k(\cdot; A),\quad
\mathscr{F}''(\cdot; A):=\Gamma\hbox{-}\ls \mathscr{F}_k(\cdot; A)
\]
for every $A\in\mathcal{A}(\Omega)$. 
If $A=\Omega$ we drop the set dependence in the above notation. Moreover we recall that if
$\mathscr{F}=\mathscr{F}'=\mathscr{F}''$ we say that $\mathscr{F}_k$ $\Gamma$-converges to $\mathscr{F}$
(with respect to the metric $d$). 

Next we recall the notion of \emph{$\overline\Gamma$-convergence},  
useful in particular to deal with the integral representation of 
$\Gamma$-limits of families of integral functionals.
We say that $(\mathscr{F}_k)$ $\overline\Gamma$
-\emph{converges} to $\mathscr{F}:X\times \mathcal{A}(\Omega)\to [0,+\infty]$ if $\mathscr{F}$ is the 
inner regular envelope of both functionals $\mathscr{F}'$ and $\mathscr{F}''$, i.e.,
\[
\mathscr{F}(u;A)=\sup\{\mathscr{F}'(u;A^\prime):\, A^\prime\in \mathcal{A}(\Omega),\, A^\prime \subset\subset A\}
=\sup \{ \mathscr{F}''(u;A^\prime): A^\prime\in \mathcal{A}(\Omega),\,A^\prime\subset\subset A\},
\]
for every $(u,A)\in X\times \mathcal{A}(\Omega)$.

\subsection{Functional setting of the problem}

All the results we shall prove in what follows will be set in the spaces $BV$ and $SBV$ 
and in suitable generalizations. For the definitions, the notations and the main properties 
of such spaces we refer to the book \cite{ambrosio}. Below we just recall the definition of 
$SBV^2(\Omega)$ that we shall often use in the sequel:
\[
SBV^2(\Omega):=
\big\{u\in SBV(\Omega):\nabla u\in L^2(\Omega)\text{ and }\mathcal{H}^{n-1}(J_u)<+\infty\big\}.
\]
Moreover, a function $u:\Omega\to\R$ belongs to $GBV(\Omega)$ (respectively to $GSBV(\Omega)$) 
if the truncations $u^M:=-M\vee(u\wedge M)$ belong to $BV_{\textsl{loc}}(\Omega)$ (respectively 
to $SBV_{\textsl{loc}}(\Omega)$), for every $M>0$. 
For fine properties of $GBV$ and $GSBV$ again we refer to \cite{ambrosio}. 

The prototype of the asymptotic result we shall prove in Sections~\ref{s:onedim}, \ref{s:ndim},
and \ref{s:further} concerns the Mumford-Shah functional of image segmentation
\be{\label{e:MS}
\MS(u):=\begin{cases}
\displaystyle{\int_{\Omega}|\nabla u|^2dx+\Hn(J_u)} & \textrm{if $u\in GSBV(\Omega)$,}\cr
+\infty                            & \textrm{otherwise in $L^1(\Omega)$}.
\end{cases}
}       
Let  $\psi:[0,1]\to[0,1]$ be any nondecreasing lower-semicontinuous function such that 
$\psi^{-1}(0)=0$ and $\psi(1)=1$. Then the classical approximation by Ambrosio and Tortorelli 
(cp. \cite{amb-tort1, amb-tort2}, and also \cite{focardi}) establishes that the two fields 
functionals $AT_k^\psi:L^1(\Omega)\times L^1(\Omega)\to[0,+\infty]$ 
\begin{equation}\label{e:ATk}
AT_k^\psi(u,v):=\begin{cases}
            \displaystyle{\int_\Omega{\Big(\psi^2(v)|\nabla u|^2+\frac{(1-v)^2}{4\ve_k}
+\ve_k|\nabla v|^2\Big) dx}} & \textrm{if $(u,v)\in \Hr$}\cr
 & \text{and $0\leq v\leq1$ $\qo$,}\\
+\infty &\text{otherwise}
           \end{cases}
\end{equation}
$\Gamma$-converge in $\Lr$ to 
\[
\MSt(u,v):=\begin{cases} 
\MS(u) & \textrm{if $v=1$ $\qo$,}\cr
+\infty & \text{otherwise},
\end{cases} 
\]
that is equivalent to the Mumford-Shah functional $\MS$ for minimization purposes.

We finally introduce the notation related to slicing.
Fixed $\xi\in \Sn:=\{\xi\in\mathbb{R}^n:|\xi|=1\}$, let $\Pi^\xi:=\big\{y\in\mathbb{R}^n:y\cdot\xi=0\big\}$, 
and for every subset $A\subset \Rn$ set 
\bes{
& A_y^\xi:=\big\{t\in\mathbb{R}:y+t\xi\in A\big\}\quad 
\text{ for $y\in \Pi^\xi$},\\
& A^\xi:=\{y\in\Pi^\xi:A^\xi_y\neq\varnothing\}.
}
For $u:\Omega\to\R$ we define the slices 
$u^\xi_y:\Omega_y^\xi\to\mathbb{R}$ by
$
u^\xi_y(t):=u(y+t\xi).
$ 

Observe that if $u_k,u\in L^1(\Omega)$ and $u_k\to u$ in $L^1(\Omega)$, 
then for every $\xi\in \Sn$ there exists a subsequence $(u_{k_j})$ such that
\[
(u_{k_j})^\xi_y\to u^\xi_y\text{ in }L^1(\Omega^\xi_y)\quad
\textrm{for $\mathcal{H}^{n-1}$-a.e.\ $y\in\Omega^\xi$}.
\]
\medskip

\section{The main results: approximation, compactness and convergence of minimizers}
\label{s:stat}
Given a  bounded open set  $\Omega\subset\Rn$ with Lipschitz boundary and   an infinitesimal sequence
 $\ve_k>0$, we  consider the sequence of functionals $F_k\colon \Lr\to [0,+\infty]$ 
\begin{equation}\label{Fk}
F_k(u,v):=\left\{ \begin{array}{ll} 
\ {\displaystyle\int_\Omega{\Big(f_k^2(v)|\nabla u|^2+\frac{(1-v)^2}{4\ve_k}
+\ve_k|\nabla v|^2\Big) dx}} 
& \textrm{if $(u,v)\in \Hr$}\\
& \text{and $0\leq v\leq1$ $\qo$},\\
\\
+ \infty & \text{otherwise},
\end{array} \right.
\end{equation}
where 
\begin{equation}\label{fk}
f_k(s):=1\wedge\ve_k^{1/2} f(s) \,,\hskip1cm f_k(1)=1\,,
\end{equation}
and 
\begin{equation}\label{f0}
\text{$f\in C^0([0,1),[0,+\infty))$ is a nondecreasing function satisfying $f^{-1}(0)=\{0\}$} 
\end{equation}
with
\be{\label{f1}
\displaystyle\lim_{s\to1}(1-s)f(s)=\ell, \quad \ell\in(0,+\infty).
}
In particular, the function $[0,1)\mapsto(1-s)f(s)$ can be continuously extended to $s=1$ with 
value $\ell$. One can consider $f(s):=\frac{s}{1-s}$ as prototype. 

It is also useful to introduce a localized version $F_k(\cdot;A)$ of $F_k$ 
simply obtained by substituting the domain of integration $\Omega$ with any measurable subset $A$
of $\Omega$ itself. In particular, to be consistent with \eqref{Fk}, for $A=\Omega$ we shall
not indicate the dependence on the domain of integration.

Let now $\Phi\colon L^1(\Omega)\to [0,+\infty]$ be defined by

\begin{equation}\label{Phi}
\Phi(u):=\left\{ \begin{array}{ll} 
\ {\displaystyle\int_\Omega h(|\nabla u|)dx+\int_{J_{u}}g(|[u]|)d\Hn+\ell|D^cu|(\Omega)} 
& \textrm{if $u\in GBV(\Omega)$,}\\
+ \infty & \text{otherwise},
\end{array} \right.
\end{equation}
where we recall that $h,g\colon[0,+\infty)\to[0,+\infty)$ are given by
\be{\label{h} h(s):=\left\{ \begin{array}{ll} 
s^2      & \textrm{if $s\leq \ell/2$,}\\
\ell s-{\ell}^2/4 & \textrm{if $s\geq \ell/2$},
\end{array} \right.}
and
\be{\label{g}
g(s):=\inf_{(\alpha,\beta)\in\mathcal{U}_s}
\int_0^1|1-\beta|\sqrt{f^2(\beta)|\alpha'|^2+|\beta'|^2}\,dt,
}
where $\mathcal{U}_s:=\mathcal{U}_s(0,1)$ and for all $T>0$
\begin{equation}\label{e:admfnctns}
 \mathcal{U}_s(0,T):=\{\alpha,\beta\in H^1\big((0,T)\big):\, 0\le \beta\le 1, \, \alpha(0)=0,\, \alpha(T)=s,\, \beta(0)=\beta(T)=1\}.
\end{equation}
At the points $t$ with $\beta(t)=1$  the integrand in (\ref{g}) reduces to $\ell |\alpha'|(t)$, in agreement with (\ref{f1}).

Our main result is the following.
\begin{theorem}\label{t:gamma-lim}
Under the assumptions above, the functionals $F_k$ $\Gamma$-converge in $\Lr$ to the functional $F$ defined by
\begin{equation}\label{F}
F(u,v):=\begin{cases} 
 \Phi(u) & \textrm{if $v=1$ $\qo$,}\cr
+ \infty & \text{otherwise}.
\end{cases} 
\end{equation}
\end{theorem}
\begin{remark}\label{r:gen}
 The assumption that $f^{-1}(0)=0$ is not restrictive and changes only the detailed properties of
 $g$. Indeed, standing all the other hypotheses, defining $\lambda:=\sup\{s\in[0,1):\,f(s)=0\}\in[0,1)$, 
 we would get that $g(s)\leq (1-\lambda)^2\wedge\ell s$ (cp.~Proposition~\ref{11} below).
 
In addition, the function $(1-v)^2$ in \eqref{Fk} can be replaced by any continuous,  decreasing
function $d(v)$ with $d(1)=0$. In this case $d^{1/2}(s)$ and $d^{1/2}(\beta)$ appear in formulas 
\eqref{f1} and \eqref{g} in place of $1-s$ and $1-\beta$ respectively, and we obtain 
$g(s)\leq 2\int_0^1d^{1/2}(t)dt\wedge\ell s$ (see  Proposition~\ref{11}).
 
Finally the definition of $f_k$ in \eqref{fk} can be given in the following more general form 
$f_k:=\psi_k\wedge\eps^{1/2}f$. Here the truncation of $f$ is performed with any continuous 
nondecreasing function $\psi_k:[0,1]\to[0,1]$ satisfying $\psi_k\geq c>0$, $\lim_k\psi_k(1)=1$, 
and converging uniformly in a neighborhood of $1$.
\end{remark}
We next address the issue of equi-coercivity for the $F_k$'s.
\begin{theorem}\label{t:comp}
Under the assumptions above, if $(u_k,v_k)\in H^1(\Omega){\times}H^1(\Omega)$ is such that
$$\sup_k \left(F_k(u_k,v_k)+||u_k||_{L^1(\Omega)}\right)<+\infty,$$
then there exists a subsequence $(u_j,v_j)$ of $(u_k,v_k)$ and a function $u\in GBV\cap L^1(\Omega)$ 
such that $u_j\to u$ $\qo$ and $v_j\to 1$ in $L^1(\Omega)$.
\end{theorem}

We shall prove Theorem~\ref{t:gamma-lim} in Sections~\ref{s:onedim} and \ref{s:ndim}, 
Theorem~\ref{t:comp} shall be established in Section~\ref{s:ndim}.

In the rest of this section instead we address the issue of convergence of minimum problems.
Minimum problems related to the functional $F_k$  could have no solution due to a lack of coercivity.
Therefore we slightly perturb the $f_k$'s to guarantee the existence of a minimum point for each $F_k$. 
This together with Theorems~\ref{t:gamma-lim} and \ref{t:comp} shall in turn imply the convergence of 
minima and minimizers as $k\uparrow\infty$.

Let $\eta_k,\ve_k$ be positive infinitesimal sequences such that $\eta_k=o(\ve_k)$ and let 
$\zeta\in L^q(\Omega)$, 
with $q>1$. Let us consider the sequence of functionals $G_k\colon \Lr\to [0,+\infty]$ defined by
$$
G_k(u,v):=\left\{ \begin{array}{ll} 
\ {\displaystyle\int_\Omega{\Big(\big(f_k^2(v)+\eta_k\big)|\nabla u|^2+\frac{(1-v)^2}{4\ve_k}
+\ve_k|\nabla v|^2+|u-\zeta|^q\Big)dx}} & \textrm{if $(u,v)\in \Hr$}\\
 & \text{and $0\leq v\leq1$ $\qo$},\\
\\
+ \infty & \text{otherwise},
\end{array} \right.
$$
where $f_k$ is as in \eqref{fk}.
Let now $\mathscr{G}\colon L^1(\Omega)\to [0,+\infty]$ be defined by

$$
\mathscr{G}(u):=\left\{ \begin{array}{ll} 
\ {\displaystyle\int_\Omega h(|\nabla u|)dx+\int_{J_{u}}g(|[u]|)d\Hn+\ell|D^cu|(\Omega)+\int_{\Omega}|u-\zeta|^qdx} 
& \textrm{if $u\in GBV(\Omega)$,}\\
+ \infty & \text{otherwise},
\end{array} \right.
$$
where $h,g,$ and $\ell$ are as in \eqref{h}, \eqref{g} and \eqref{f1} respectively.
Then the following corollary holds true.
\begin{corollary}\label{c:min}
For every $k$,
let $(u_k,v_k)\in H^1(\Omega){\times} H^1(\Omega)$ be a minimizer of the problem
\begin{equation}\label{intro43}
\min_{(u,v)\in H^1(\Omega){\times} H^1(\Omega)}G_k(u_k,v_k).
\end{equation}
Then $v_k\to 1$ in $L^1(\Omega)$ and a subsequence of $u_k$ converges in $L^q(\Omega)$
to a minimizer $u$ of the problem
$$
\min_{u\in GBV(\Omega)}\mathscr{G}(u).
$$
Moreover the minimum values of (\ref{intro43}) tend to the minimum value of the limit problem.
\end{corollary}
\begin{proof}
We shall only sketch the main steps to establish the conclusion, being the arguments quite standard. 

One first proves that in fact the functionals $F_k$ $\Gamma$-converge to $F$ in 
$L^q(\Omega){\times}L^1(\Omega)$, where $F_k$ and $F$ are the
functionals in Theorem~\ref{t:gamma-lim}.
Indeed, when $q>1$ the $\Gamma$-limsup inequality works exactly as in the case $q=1$, 
whereas the $\Gamma$-liminf inequality is an immediate consequence of the comparison 
with the case $q=1$ and of \cite[Proposition 6.3]{dalmaso}. Let us observe that the presence of $\eta_k$ 
in the functional $F_k$ does not modify the $\Gamma$-convergence result and that the proofs still hold analogously.

As a consequence $G_k$ $\Gamma$-converges to $G$ in  $\Lr$ for every $q\geq 1$, 
where $G(u,v):=\mathscr{G}(u)$ if $v=1$ $\qo$ and $\infty$ otherwise.
Indeed, \cite[Proposition 6.3]{dalmaso} yields that the $\Gamma$-limsup
of $G_k$ in $\Lr$ is less than or equal to the one in $L^q(\Omega){\times}L^1(\Omega)$. In addition, 
$\int_{\Omega}|\cdot-\zeta|^qdx$ is continuous in $L^q(\Omega){\times}L^1(\Omega)$, so that the conclusion 
follows from \cite[Propositions 6.17 and 6.21]{dalmaso}.

The previous result combined with a general result of $\Gamma$-convergence technique 
\cite[Corollary 7.20]{dalmaso} concludes the proof of Corollary \ref{c:min} through 
the compactness result Theorem \ref{t:comp}.
\end{proof}

\section{Properties of the surface energy density}\label{s:gprop}
In this section we shall establish several properties enjoyed by the surface energy density $g$ 
defined in \eqref{g}. 

To this aim we shall often exploit that, in computing 
$g(s)$, $s\geq0$, we may assume that the admissible functions $\alpha$ satisfy $0\leq\alpha\leq s$ by a 
truncation argument (whereas $0\leq\beta\leq1$ by definition). Further, given a curve $(\alpha,\beta)\in  \mathcal{U}_s(0,T)$, 
note that the integral appearing in the definition of $g$ is invariant under reparametrizations of $(\alpha,\beta)$. 
\begin{proposition}\label{11}
The function $g$ defined in (\ref{g}) enjoys the following properties:
\begin{enumerate}[label=(\roman*)]
\item\label{e21} $g(0)=0$, and $g$ is subadditive, \textsl{i.e.}, 
$g(s_1+s_2)\leq g(s_1)+g(s_2)$, for every $s_1,s_2\in\R^+$;
\item\label{e19} $g$ is nondecreasing, $0\leq g(s)\leq 1\wedge\ell s$ for all $s\in\R^+$,
and $g$ is Lipschitz continuous with Lipschitz constant $\ell$;
\item\label{e20} 
\be{\label{e23}\lim_{s\uparrow\infty}g(s)=1;}
\item\label{e22} 
\be{\label{e24}\lim_{s\downarrow 0}\frac{g(s)}{s}=\ell.}
\end{enumerate} 
\end{proposition}
\begin{proof}
Proof of \ref{e21}.
 The couple $(\alpha,\beta)=(0,1)$ is admissible for the minimum problem defining $g(0)$, so that $g(0)=0$.

In order to prove that $g$ is subadditive we fix $s_1,s_2\in\R^+$ and we consider the minimum problems for 
$g(s_1)$ and $g(s_2)$, respectively. Let $\eta>0$ and let $(\alpha_1,\beta_1),(\alpha_2,\beta_2)$ be admissible 
couples respectively for $g(s_1)$ and $g(s_2)$ such that for $i=1,2$ 
\be{\label{e27}
\int_0^1|1-\beta_i|\sqrt{f^2(\beta_i)|\alpha_i'|^2+|\beta_i'|^2}dt<g(s_i)+\eta.
} 
Next define $\alpha:=\alpha_1$ in $[0,1]$, $\alpha:=\alpha_2(\cdot-1)+s_1$ in $[1,2]$, $\beta:=\beta_1$ in 
$[0,1]$, and $\beta:=\beta_2(\cdot-1)$ in $[1,2])$. 
An immediate computation and the reparametrization property mentioned above 
entail the subadditivity of $g$ since $\eta$ is arbitrary. 

Proof of \ref{e19}.
In order to prove that $g$ is nondecreasing we fix $s_1,s_2$ with $s_1< s_2$ and $\eta>0$, and we consider 
$(\alpha,\beta)$ satisfying a condition analogous to
\eqref{e27} for $g(s_2)$. 
Then $(\frac{s_1}{s_2}\alpha,\beta)$ is admissible for $g(s_1)$, thus we infer
$$g(s_1)\leq \int_0^1|1-\beta|\sqrt{\Big(\frac{s_1}{s_2}\Big)^2f^2(\beta)|\alpha'|^2+|\beta'|^2}dt< g(s_2)+\eta,$$
since $s_1/{s_2}< 1$. As $\eta\to 0$ we find $g(s_1)\leq g(s_2)$.

Next we prove that $g(s)\leq 1\wedge \ell s$.
Indeed, inequality $g\leq 1$ straightforwardly comes from the fact that for every $s\geq 0$ the following couple 
$(\alpha,\beta)$ is admissible for $g(s)$: 
$\alpha:=0$ in $(0,1/3)$, $\alpha:=s$ in $(2/3,1)$, and linearly linked in $(1/3,2/3)$, and $\beta:=0$ in $(1/3,2/3)$ and linearly
linked to $1$ in $(0,1/3)$ and in $(2/3,1)$.
Moreover, $g(s)\leq \ell s$ for every $s\geq0$ since the couple $(st,1)$ is admissible for $g(s)$.

The Lipschitz continuity of $g$ is an obvious consequence of the facts 
that $g$ is nondecreasing, subadditive and $g(s)\leq \ell s$ for $s\geq0$. 

Proof of \ref{e20}. 
Let $s_k$, $k\in\N$, be a diverging sequence and let 
$(\alpha_k,\beta_k)$ be an admissible couple for $g(s_k)$ such that
\be{\label{e25}\int_0^1|1-\beta_k|\sqrt{f^2(\beta_k)|\alpha'_k|^2+|\beta'_k|^2}dt<g(s_k)+\frac{1}{k}.}
If $\inf_{(0,1)}\beta_k\geq\delta$ for some $\delta>0$ and for every $k$, then there exists a constant $c(\delta)>0$ such that  
$f(\beta_k)(1-\beta_k)>c(\delta)$, since $f(s)(1-s)\to 0$ if and only if $s\to0$.
Therefore by \eqref{e25} one finds
$$c(\delta)s_k\leq g(s_k)+\frac{1}{k},$$
so that $g(s_k)\to+\infty$ as $k\to+\infty$ and this contradicts the fact that $g\leq 1$. Therefore there exists a sequence $x_k\in(0,1)$
such that $\beta_k(x_k)\to0$ up to subsequences. Since we have already shown that $g\leq 1$, we conclude the proof of \eqref{e23} noticing 
that \eqref{e25} yields \be{\label{e26}(1-\beta_k(x_k))^2 \leq\int_0^{x_k}|1-\beta_k||\beta'_k|dt+\int_{x_k}^1|1-\beta_k||\beta'_k|dt\leq g(s_k)+\frac{1}{k}.}

Proof of \ref{e22}. 
Let $s_k$, $k\in\N$, be an infinitesimal sequence and let $(\alpha_k,\beta_k)$ be an admissible 
couple for $g(s_k)$ satisfying \eqref{e25} with $s_k/k$ in place of $1/k$. If there exists $\delta>0$, 
a not relabeled subsequence of $k$, and a sequence $x_k\in[0,1]$ such that $\beta_k(x_k)<1-\delta$, 
then the same computation as in \eqref{e26} leads to 
\[\delta^2\leq g(s_k)+\frac{s_k}{k}.
\]
As $k\to+\infty$ this contradicts the fact that $g(s)\le \ell s$. Therefore, $\beta_k$ converges uniformly to $1$ and fixing $\delta>0$
$$(\ell-\delta)s_k\leq\int_0^1(1-\beta_k)f(\beta_k)|\alpha'_k|dt\leq g(s_k)+\frac{s_k}{k}$$
holds for $k$ large by \eqref{f1}. Formula \eqref{e24} immediately follows dividing both sides of the last inequality by $s_k$, taking
first $k\to+\infty$ and then $\delta\to0$, and using the fact that $g(s)\leq \ell\,s$, for $s\geq0$.   
\end{proof}
\begin{remark}\label{r:g}
We can actually show that $g$ does not coincide with the function $1\wedge\ell\,s$ at least in the model case 
$f(s)=\frac{\ell s}{1-s}$ by slightly refining the construction used in \ref{e19} above. 
With fixed $s>0$, let $\lambda\in[0,1]$ and set $\alpha:=0$ on $[0,1/3]$, $\alpha:=s$ on $[1/3,2/3]$, and
the linear interpolation of such values on $[1/3,2/3]$; moreover, set $\beta_\lambda:=\lambda$ on $[1/3,2/3]$
and the linear interpolation of the values $1$ and $\lambda$ on each interval $[0,1/3]$ and $[2/3,1]$ in order
to match the boundary conditions. Straightforward calculations lead to
\[
g(s)\leq(1-\lambda)^2+(1-\lambda)f(\lambda)\,s.
\]
Thus, minimizing over $\lambda\in[0,1]$ yields in turn
\[
g(s)\leq \ell s-\frac{(\ell s)^2}{4}<1\wedge\ell s\quad\text{ for all $s\in(0,2/\ell)$.}
\]
\end{remark}

In what follows it will be convenient to provide an alternative representation of $g$ by means of 
a cell formula more closely related to the one-dimensional version of the energies $F_k$'s. 

To this aim we introduce the function $\hat{g}\colon [0,+\infty)\to[0,+\infty)$ defined by
\begin{equation}\label{hatg}
\hat{g}(s):=\lim_{T\uparrow\infty}\inf_{(\alpha,\beta)\in\mathcal{U}_s(0,T)}
\int_0^T \left(f^2(\beta)|\alpha'|^2+\frac{|1-\beta|^2}{4}+|\beta'|^2\right)dt,
\end{equation}
the class $\mathcal{U}_s(0,T)$ has been introduced in \eqref{e:admfnctns}.
We note that $\hat{g}$ is well-defined as the minimum problems appearing in its definition are decreasing 
with respect to $T$.
\begin{proposition}\label{p2}
For all $s\in [0,+\infty)$ it holds $g(s)=\hat{g}(s)$. 
\end{proposition}
\begin{proof}
  Let $\alpha,\beta \in H^1\big((0,T)\big)$, $T>0$, be admissible functions in the definition of $\hat g(s)$. 
By Cauchy inequality we obtain
\begin{equation*}
\sqrt{f^2(\beta)|\alpha'|^2 + |\beta'|^2}\, |1-\beta|\le f^2(\beta)|\alpha'|^2 + |\beta'|^2+\frac{(1-\beta)^2}{4} 
\end{equation*}
and integrating
\begin{equation*}
\int_0^T |1-\beta| \, \sqrt{f^2(\beta)|\alpha'|^2 + |\beta'|^2}\, dt\le \int_0^T 
 \left(f^2(\beta)|\alpha'|^2+\frac{|1-\beta|^2}{4}+|\beta'|^2\right)dt\,.
\end{equation*}
The first integral is one-homogeneous in the derivatives, therefore we can reparametrize from $(0,T)$ 
to $(0,1)$. Taking the infimum over all such $\alpha$, $\beta$, and $T$ we obtain $g(s)\le \hat g(s)$.   

To prove the converse inequality, we first show that
$\alpha$ and $\beta$ in the infimum problem defining $g$ can be taken in $W^{1,\infty}\big((0,1)\big)$. Let $\eta>0$ small and let $\alpha,\beta\in H^1\big((0,1)\big)$
be competitors for $g(s)$ such that
\be{\label{e:quasimin}\int_{0}^{1} |1-\beta| \, \sqrt{f^2(\beta)|\alpha'|^2 + |\beta'|^2}\, dt< g(s)+\eta.}
By density we find two sequences $\alpha_j,\beta_j\in W^{1,\infty}\big((0,1)\big)$ (actually in $C^\infty([0,1])$) such that
$\alpha_j(0)=0$, $\alpha_j(1)=s$, $\beta_j(0)=\beta_j(1)=1$, $0\leq\beta_j\leq 1$, and converging respectively 
to $\alpha$ and $\beta$ in $H^1\big((0,1)\big)$. Since the function $(1-s)f(s)$ is uniformly continuous and $\beta_j\to\beta$ also uniformly, we deduce that
$$\int_{0}^{1} |1-\beta_j| \, \sqrt{f^2(\beta_j)|\alpha_j'|^2 + |\beta_j'|^2}\, dt< g(s)+\eta$$
for $j$ large, and this concludes the proof of the claim.

Let us prove now that $\hat{g}\leq g$. 
We fix a small parameter $\eta>0$ and
 consider competitors $\alpha,\beta\in W^{1,\infty}\big((0,1)\big)$  
for $g(s)$ satisfying \eqref{e:quasimin}.
We  define, for $t\in [0,1]$,
\begin{equation*}
  \beta^\eta(t):=\beta(t)\wedge (1-\eta)\quad \text{and}\quad \psi_\eta(t):=\int_{0}^t \frac{2}{1-\beta^\eta} 
  \sqrt{\eta+f^2(\beta^\eta)\,  |\alpha'|^2 + |(\beta^\eta)'|^2} dt' \,.
\end{equation*}
The function  $\psi_\eta:[0,1]\to[0,M_\eta:=\psi_\eta(1)]$ is bilipschitz and in particular invertible. 
We define  $\bar\alpha^\eta, \bar\beta^\eta\in W^{1,\infty}\big((0,M_\eta)\big)$ by
\begin{equation*}
  \bar\alpha^\eta := \alpha \circ \psi_\eta^{-1} \text{ and }
  \bar\beta^\eta := \beta^\eta \circ \psi_\eta^{-1}\,.
\end{equation*}
We compute, using the definition and the change of variables $x=\psi_\eta(t)$,
\begin{alignat*}1
  \int_0^{M_\eta} \frac{(1-\bar\beta^\eta)^2}4 dx &= 
  \int_0^{M_\eta} \frac{(1-\beta^\eta(\psi_\eta^{-1}(x)))^2}4 dx = 
  \int_{0}^1 \frac{(1-\beta^\eta(t))^2}4 \psi_\eta'(t) dt \\
&=\int_{0}^1 \frac{1-\beta^\eta}2 \sqrt{\eta+
f^2(\beta^\eta)\,  |\alpha'|^2 + |(\beta^\eta)'|^2} dt \\
&\le \sqrt\eta+ \int_{0}^1 \frac{1-\beta^\eta}2 \sqrt{f^2(\beta^\eta)\,  |\alpha'|^2 + |(\beta^\eta)'|^2} dt\,,
\end{alignat*}
where we inserted $\psi_\eta'$ from the definition of $\psi_\eta$ and used 
$\sqrt{\eta+A}\le \sqrt \eta + \sqrt A$. Analogously,
\begin{alignat*}1
  \int_0^{M_\eta} \left( f^2(\bar\beta^\eta)|(\bar\alpha^\eta)'|^2+|(\bar\beta^\eta)'|^2\right) dx &= 
  \int_{0}^1 \left( f^2(\beta^\eta)|\alpha'|^2+|(\beta^\eta)'|^2\right) \frac{1}{\psi'_\eta} dt \\
&= \int_{0}^1 \left( f^2(\beta^\eta)|\alpha'|^2+|(\beta^\eta)'|^2\right) 
 \frac{1-\beta^\eta}{2 \sqrt{\eta+
f^2(\beta^\eta)\,  |\alpha'|^2 + |(\beta^\eta)'|^2}}dt  \\
&\le \int_{0}^1
 \frac{1-\beta^\eta}{2} \sqrt{f^2(\beta^\eta)\,  |\alpha'|^2 + |(\beta^\eta)'|^2}dt \,.
\end{alignat*}
We extend $\bar\alpha^\eta$ and $\bar\beta^\eta$ to $(-1, M_\eta+1)$ setting 
 $\bar\alpha^\eta:=0$ in $(-1,0)$, $\bar\alpha^\eta:=s$ in $(M_\eta,M_\eta+1)$, and $\bar\beta^\eta$ the linear interpolation between $1-\eta$ and 1 in each of the 
two intervals, so that they obey the required boundary conditions  for $\hat g$ in the larger interval. 
Collecting terms, we obtain
\begin{align}
  \hat g(s)&\le  \int_{-1}^{M_\eta+1} \left(\frac{(1-\bar\beta^\eta)^2}4 
+ f^2(\bar\beta^\eta)|\bar\alpha'_\eta|^2+|\bar(\beta^\eta)'|^2\right) dx\nonumber\\
&\le \sqrt \eta + 3\eta^2 +  \int_{0}^1
(1-\beta^\eta) \sqrt{f^2(\beta^\eta)\,  |\alpha'|^2 + |(\beta^\eta)'|^2}dt \,,\label{al:p1}
\end{align}
where the $3\eta^2$ term comes from an explicit computation on the two boundary intervals.

It remains to replace $\beta^\eta$ by $\beta$ in the last integral. 
We observe that  $(\beta^\eta)'=0$ almost everywhere on the set where $\beta\ne \beta^\eta$ (which coincides with the set $\{\beta>1-\eta\}$). 
 Therefore
\begin{alignat*}1
\int_{\{\beta\ne \beta^\eta\}}
(1-\beta^\eta) \sqrt{f^2(\beta^\eta)\,  |\alpha'|^2 + |(\beta^\eta)'|^2}dt &=
\int_{\{\beta\ne \beta^\eta\}}
(1-\beta^\eta) f(\beta^\eta)\,  |\alpha'| dt \\
&\le \int_{\{\beta\ne \beta^\eta\}}
(1-\beta) f(\beta)\,  |\alpha'| dt 
+ \omega(\eta) \int_0^1 |\alpha'| dt
\end{alignat*}
where $\omega(\eta)$ is the continuity modulus of $(1-s)f(s)$ near $s=1$, and therefore
\begin{alignat*}1
  \hat g(s)
&\le \sqrt \eta + 3\eta^2  
+\omega(\eta)\int_0^1|\alpha'| dt+
 \int_{0}^1
(1-\beta) \sqrt{f^2(\beta)\,  |\alpha'|^2 + |\beta'|^2}dt \,.
\end{alignat*}
Since the last integral is less than $g(s)+\eta$ and $\eta$ can be made arbitrarily small, this concludes the proof.
\end{proof}

For the proof of the lower bound
we also need to introduce the auxiliary functions $g^{(\eta)}\colon [0,+\infty)\to[0,+\infty)$,
for $\eta>0$, defined by
\be{\label{geta}
g^{(\eta)}(s):=\inf_{(\alpha,\beta)\in\mathcal{U}^{(\eta)}_s}
\int_0^1|1-\beta|\sqrt{f^2(\beta)|\alpha'|^2+|\beta'|^2}\,dt,
}
where
$$
 \mathcal{U}^{(\eta)}_s:=\{\alpha,\beta\in H^1\big((0,1)\big):\, \alpha(0)=0,\, \alpha(1)=s,\, \beta(0)=\beta(1)=1-\eta\}.
$$
\begin{proposition}\label{p1}
For all $s\in [0,+\infty)$ it holds 
\[
|g(s)-g^{(\eta)}(s)|\leq\eta^2. 
\]
\end{proposition}
\begin{proof}
We consider the minimum problems for $g$ and $g^{(\eta)}$ respectively in the intervals $(-1,2)$ and $(0,1)$.
Let $(\alpha_{\eta},\beta_{\eta})$ be an admissible couple for $g^{(\eta)}(s)$ and let $\alpha:=0$ in $(-1,0)$, 
$\alpha:=\alpha_{\eta}$  in $(0,1)$, and $\alpha:=s$ in $(1,2)$;
we also set $\beta:=\beta_{\eta}$ in $(0,1)$ and linearly linked to $1$ in $(-1,0)$ and in $(1,2)$. 
Then an easy computation shows that
$$g(s)\leq \int_0^1|1-\beta_{\eta}|\sqrt{f^2(\beta_{\eta})|\alpha'_{\eta}|^2+|\beta'_{\eta}|^2}dt+\eta^2.$$
By taking the infimum on $(\alpha_{\eta},\beta_{\eta})$ we infer that 
$$g(s)\leq g^{(\eta)}(s)+\eta^2.$$
Reversing the roles of $g$ and $g^{(\eta)}$ we conclude.
\end{proof}

Finally, we study the dependence of $g$ on the function $f$ in detail. The results in the next proposition
provide a first insight on the class of functions $g$ that arise as surface energy densities in our 
analysis. Moreover, they will be instrumental to get in the limit different energies by slightly changing
the functionals $F_k$'s in \eqref{Fk} (cp. Theorems~\ref{t:Dugdale}, \ref{t:sublin}, and \ref{t:MS} below).
\begin{proposition}\label{p:gl}
Let $(\f{j})$ be a sequence of functions satisfying \eqref{f0} and \eqref{f1}. Denote by $\ell_j$, $g_j$
the value of the limit in \eqref{f1} and the function in \eqref{g} corresponding to $\f{j}$, respectively. 
Then,
 \begin{itemize}
  \item[(i)] if $\ell_j=\ell$ for all $j$, $\f{j}\geq \f{j+1}$, and $\f{j}(s)\downarrow 0$ for all 
  $s\in[0,1)$, then $g_j\geq g_{j+1}$ and $g_j(s)\downarrow 0$ for all $s\in[0,+\infty)$;
  
  \item[(ii)] if $\ell_j=\ell$ for all $j$, $\f{j}\leq \f{j+1}$, and $\f{j}(s)\uparrow \infty$ for all 
  $s\in(0,1)$, then $g_j\leq g_{j+1}$ and $g_j(s)\uparrow 1\wedge\ell s$ for all $s\in[0,+\infty)$;
  
  \item[(iii)] if $\ell_j\uparrow\infty$, $\f{j}\leq \f{j+1}$, and $\f{j}(s)\uparrow \infty$ for all 
  $s\in(0,1)$, then $g_j\leq g_{j+1}$ and $g_j(s)\to \chi_{(0,+\infty)}(s)$ for all $s\in[0,+\infty)$. 
 \end{itemize}
\end{proposition}
\begin{proof}
 To prove item (i) we note that the monotonicity of the sequence $(\f{j})$ and the 
 pointwise convergence to a continuous function on $[0,1)$ yield that the sequence
 $(\f{j})$ actually converges uniformly on compact subsets of $[0,1)$ to $0$. 
 Therefore, for all $\delta\in(0,1)$ we have for some $j_\delta$
 \[
 \max_{[0,1-\delta]}\f{j}\leq\delta\qquad\text{for all $j\geq j_\delta$.}
 \]
 Then, consider $\alpha_j,\beta_j$ defined as follows:  $\alpha_j(t):=3s(t-1/3)$ on 
 $[1/3,2/3]$, $\alpha_j:=0$ on $[0,1/3]$, and $\alpha_j:=s$ on $[2/3,1]$; $\beta_j:=1-\delta$ 
 on $[1/3,2/3]$ and a linear interpolation between the values $1$ and $1-\delta$ on each interval 
 $[0,1/3]$ and $[2/3,1]$. Straightforward calculations give 
 \[
 g_j(s)\leq\delta^2\,s+\delta^2 \qquad\text{for all $j\geq j_\delta$,}
 \]
 from which the conclusion follows by passing to 
 the limit first in $j\uparrow\infty$ and finally letting $\delta\downarrow 0$.
 
 We now turn to item (ii). We first note that $(g_j)$ is nondecreasing and that 
 \begin{equation}\label{e:geasy}
 \lim_jg_j(s)\leq \ell s\wedge 1
 \end{equation}
 in view of item \ref{e19} in Proposition~\ref{11}.
 Next we show the following: for all $\delta>0$ 
 \begin{equation}\label{e:dmlim}
  \lim_j\min_{t\in[\delta,1]}(1-t)\f{j}(t)=\ell.
 \end{equation}
Let $s_j\in\textrm{argmin}_{[\delta,1]}(1-t)\f{j}(t)$, and denote by $j_k$
a subsequence such that 
\[
\lim_k\min_{t\in[\delta,1]}(1-t)\f{j_k}(t)=\liminf_j\min_{t\in[\delta,1]}(1-t)\f{j}(t).
\]
Either $\limsup_ks_{j_k}<1$ or $\limsup_ks_{j_k}=1$. We exclude the former possibility: 
suppose that, up to further subsequences not relabeled, $\lim_ks_{j_k}=s_\infty\in[\delta,1)$, 
then for all $i\in\N$
\[
\liminf_k(1-s_{j_k})\f{j_k}(s_{j_k})=(1-s_\infty)\liminf_k\f{j_k}(s_{j_k})
\geq(1-s_\infty)f^{(i)}(s_\infty), 
\]
that gives a contradiction by letting $i\uparrow\infty$ since  by minimality of $s_j$ 
\be{\label{e:boundell}
(1-s_j)\f{j}(s_j)\leq\ell\qquad\text{for all $j$}.
}
Therefore, $\limsup_ks_{j_k}=1$, and thus we get
\[
\liminf_j(1-s_j)\f{j}(s_j)\geq \liminf_k(1-s_{j_k})\f{1}(s_{j_k})=\ell.
\]
Formula \eqref{e:dmlim} follows straightforwardly by this and \eqref{e:boundell}.

If $s=0$, clearly we conclude as $g_j(0)=0$ for all $j$.
Let then $s\in(0,+\infty)$ and $\alpha_j$, $\beta_j\in H^1\big((0,1)\big)$ be such that 
$\alpha_j(0)=0$, $\alpha_j(1)=s$, $\beta_j(0)=\beta_j(1)=1$ and
\[
g_j(s)+\frac1j\geq
\int_0^1|1-\beta_j|\sqrt{(\f{j})^2(\beta_j)|\alpha'_j|^2+|\beta'_j|^2}dt.
\]
There are now two possibilities: either there exists $\delta>0$ and a subsequence $j_k$ such 
that $\inf_{[0,1]}\beta_{j_k}\geq\delta$, or $\inf_{[0,1]}\beta_j\to0$. In the former case the 
subsequence satisfies
\begin{equation*}
g_{j_k}(s)+\frac1{j_k}\geq\big(\min_{t\in[\delta,1]}(1-t)\f{j_k}(t)\big)\,s.
\end{equation*}
Taking the $\limsup_{k}$ and using (\ref{e:dmlim}) we 
obtain
\begin{equation}\label{eqcaso1}
  \limsup_j g_j(s) \ge \ell s\,.
\end{equation}
In the other case for every $\delta>0$ definitively it holds
\begin{equation}\label{e:altb}
g_{j}(s)+\frac1{j}\geq\int_0^1(1-\beta_{j})|\beta_{j}^\prime|dt\geq (1-\delta)^2.
\end{equation}
Taking again the $\limsup$  we obtain
\begin{equation}\label{eqcaso2}
  \limsup_j g_j(s) \ge (1-\delta)^2\,.
\end{equation}
Since $\delta$ was arbitrary,
from (\ref{eqcaso1}) and (\ref{eqcaso2}) we obtain $\limsup_j g_j(s)\ge 1\wedge \ell s$ and,
recalling (\ref{e:geasy}),  conclude the proof of (ii).

Let us now prove item (iii). 
First we observe that $g_j(s)\leq 1$ for all $j$.
To prove the lower bound, we notice that arguing similarly as in the proof of  \eqref{e:dmlim}
one obtains
\begin{equation}\label{e:dmlim2}
  \lim_j\min_{t\in[\delta,1]}(1-t)\f{j}(t)=\infty \text{ for all $\delta>0$}.
 \end{equation}
For any  $s\in(0,+\infty)$ we choose $(\alpha_j, \beta_j)\in \mathcal{U}_s$ such that
\[
g_j(s)+\frac1j\geq\int_0^1|1-\beta_j|\sqrt{(\f{j})^2(\beta_j)|\alpha'_j|^2+|\beta'_j|^2}dt.
\]
If there is $\delta>0$ such that $\inf \beta_j\ge \delta$ for infinitely many $j$ then for the same
indices
\begin{equation*}
  g_j(s)+\frac1j\ge \min_{t\in[\delta,1]} (1-t) f^{(j)}(t) s \,,
\end{equation*}
which in view of (\ref{e:dmlim2}) and the bound $g_j(s)\le 1$ is impossible.
Therefore  $\inf_{[0,1]}\beta_j\to0$, which in view of (\ref{e:altb}) proves the assertion.
\end{proof}
\begin{remark}
 The monotonicity assumption $\f{j}\leq \f{j+1}$ in items (ii) and (iii) above leads to simple proofs but 
 it is actually not needed. The same convergence results for $(g_j)$ would follow by using the uniform 
 convergence on compact subsets of $[0,1)$ of $(\f{j})$. The latter property is a consequence of the fact 
 that each $\f{j}$ is nondecreasing and that $f\in C^0([0,1))$. 
\end{remark}

\section{Proof in the one-dimensional case}\label{s:onedim}

Let us study first the one-dimensional case $n=1$. As usual, we will prove a $\Gamma$-liminf inequality and a $\Gamma$-limsup inequality.
The following proposition gives the lower estimate.

\begin{proposition}[Lower bound]\label{p:liminfunidim} 
For every $(u,v)\in \Lr$ it holds
$$F(u,v)\leq F'(u,v).$$
\end{proposition}
\begin{proof}
The conclusion is equivalent to the following fact:
let $(u_k,v_k)$ be a sequence such that
\begin{equation}\label{ukvk}
(u_k,v_k)\to (u,v) \text{ in } L^1(\Omega){\times} L^1(\Omega),
\end{equation}
\begin{equation}\label{bounded}
\sup_k F_{k}(u_k,v_k)<+\infty,
\end{equation}
then $u\in BV(\Omega)$, $v=1$ $\mathcal{L}^1$-a.e.\ in $\Omega$,
and
\be{\label{1}\Phi(u)\leq
\liminf_{k\to \infty}F_k(u_k,v_k).}
Since the left-hand side of \eqref{1} is $\sigma$-additive and the right-hand side is $\sigma$-superadditive 
with respect to $\Omega$, it is enough to prove the result when $\Omega$ is an interval. 
For the sake of convenience in what follows we assume $\Omega=(0,1)$. 

By \eqref{bounded} one deduces that $v=1$ $\qo$. Up to subsequences one can assume that the lower limit in \eqref{1} is in fact a limit
and that the convergences in \eqref{ukvk} are also $\mathcal{L}^1$-a.e.\ in $\Omega$.

For the first part of the proof we will use a discretization argument, following the lines of \cite{AlicBrShah}.
We fix  $\delta\in(0,1)$ and for any 
 $N\in\N$  divide $\Omega$ into $N$ intervals
$$I^j_N:=\Big(\frac{j-1}{N},\frac{j}{N}\Big),\qquad j=1,\dots,N.$$
Up to subsequences we can assume that $\displaystyle\lim_{k\to+\infty} \inf_{I^j_N} v_k$ exists for every $j=1,\dots,N$.
We define
$$J_N:=\Big\{j\in\{1,\dots,N\}:\lim_{k\to+\infty} \inf_{I^j_N} v_k\le 1-\delta\Big\}.$$
Fixed $j\in J_N$, we denote by $x_k$ and $y$ two points in $I_N^j$ such that $v_k(x_k)<1-\delta/2$ and  $v_k(y)\to 1$. Then by Cauchy's inequality we 
deduce for $k$ large (assuming for instance $x_k\leq y$)
\be{\label{2}\int_{x_k}^{y}\Big(\frac{(1-v_k)^2}{4\ve_k}+\ve_k| v'_k|^2\Big) dx\geq \frac{1}{2}((1-v_k(x_k))^2-(1-v_k(y))^2)\geq \frac{\delta^2}{16}.}
The previous computation entails 
$$\sup_N\Ho(J_N)<+\infty,$$
so that up to subsequences we can assume $J_N=\{j^N_1,\dots,j^N_L\}$, with $L$ independent on $N$,
and that all sequences $j^N_i/N$ converge. We denote by $S$ the set
of limits of these sequences,
$$
S=\{t_1,\dots,t_{L'}\}=\bigl\{\lim_{N\to+\infty}\frac{j^N_i}{N}\,,\hskip2mm i=1,\dots,L\bigr\} \subset\Omega\,.
$$
We claim now that there exists a modulus of continuity $\omega$, i.e., $\omega(\delta)\to0$ as $\delta\to0$, 
depending only on $f$, such that for all $\eta$ sufficiently small and  $k$ sufficiently large 
(depending on $\eta$)  one has
\be{\label{3}
(1-\omega(\delta))\int_{\Omega\setminus S_{\eta}}h(|u'_k|)dx\leq F_k(u_k,v_k,\Omega\setminus S_{\eta}),
}
where $S_{\eta}:=\bigcup_{i=1}^{L'}(t_i-\eta,t_i+\eta)$. 
It suffices to prove  (\ref{3}) in the case that  $\eta$ is so small that 
 the intervals $(t_i-\eta,t_i+\eta)$  are pairwise disjoint.

In order to prove \eqref{3}, we observe that by definition of $f_k$ in \eqref{fk} and by Cauchy's inequality we obtain
\ba{\label{4}
F_k(u_k,v_k;\Omega\setminus S_{\eta})&\geq& \int_{\Omega\setminus S_{\eta}}\Big(f_k^2(v_k)|u'_k|^2+\frac{(1-v_k)^2}{4\ve_k}\Big) dx\nonumber\\
&\geq& \int_{\Omega\setminus S_{\eta}}\Big(|u'_k|^2\wedge
\bigl(\ve_kf^2(v_k)|u'_k|^2+\frac{(1-v_k)^2}{4\ve_k}\bigr)\Big)dx\nonumber\\
&\geq& \int_{\Omega\setminus S_{\eta}}|u'_k|^2\wedge\bigl((1-v_k)f(v_k)|u'_k|\bigr)dx.
}
Let us note that $v_k> 1-\delta$ in $\Omega\setminus S_{\eta}$ for $k$ large. By \eqref{f1} there exists a modulus of continuity $\omega$ such that
\be{\label{5}|(1-s)f(s)-\ell|\leq \ell\omega(\delta),\qquad \textrm{for $s\geq1-\delta$}.} 
Therefore by \eqref{4} and \eqref{5} we obtain
\be{\label{6}
F_k(u_k,v_k;\Omega\setminus S_{\eta})\geq(1-\omega(\delta))\int_{\Omega\setminus S_{\eta}}|u'_k|^2\wedge\ell|u'_k|dx
\geq(1-\omega(\delta))\int_{\Omega\setminus S_{\eta}}h(|u'_k|)dx.
}
The last inequality holds true as $h$ is the convex envelope of $t^2\wedge\ell t$. Formula (\ref{6}) proves the claim in \eqref{3}.

Notice that the boundedness assumption in \eqref{bounded} and formula \eqref{3} imply that 
$$\sup_{k}\int_{\Omega\setminus S_{\eta}}|u'_k|dx<+\infty.$$
Therefore $u\in BV(\Omega\setminus S_{\eta})$, and actually the finiteness of $S$ ensures that $u\in BV(\Omega)$. In addition, the
$L^1$-lower semicontinuity of the functional $\Phi$ defined in \eqref{Phi} yields
\be{\label{e16}(1-\omega(\delta))\Phi(u;\Omega\setminus S_{\eta})\leq\liminf_k F_k(u_k,v_k;\Omega\setminus S_{\eta}).}

We now estimate the energy contribution on $S_{\eta}$. To this aim it is not restrictive to assume that $S\subseteq J_u$.
\medskip

Let us fix $i\in\{1,\dots,L'\}$ and consider $I^i_\eta:=(t_i-\eta,t_i+\eta)$. We claim that
\be{\label{8}(1-\omega(\delta))g(\esssup_{I^i_{\eta}} u-\essinf_{I^i_{\eta}} u)\leq\li F_k(u_k,v_k;I^i_{\eta})+O(\eta).}

Let us introduce a small parameter $\mu>0$ and $x_1,x_2\in I^i_{\eta}$ such that
\ba{ &\displaystyle v_k(x_1)\to 1, \qquad v_k(x_2)\to 1, \nonumber\\ 
\label{13} &\displaystyle u_k(x_1)\to u(x_1), \qquad u_k(x_2)\to u(x_2), \\
\label{7} &\displaystyle u(x_1)>\esssup_{I^i_{\eta}} u-\mu,\qquad u(x_2)<\essinf_{I^i_{\eta}} u+\mu.}
Assuming without loss of generality that $x_1< x_2$, we define $I:=(x_1,x_2)$. 

There are just finitely many connected components of the set
$$\{x\in I: v_k(x)< 1-\eta\}$$
where $v_k$ achieves the value $1-\delta$, as a computation analogous to \eqref{2} easily shows (recall that $\eta\ll\delta$). Precisely one finds up to subsequences
that the number $N$ of these components is $$N\leq\frac{c}{\delta^2-\eta^2},$$
for some constant $c>0$ independent of $N$. 
Let us now estimate the functional $F_k$ over each component $C_k^j$ of this type, $j=1,\dots,N$.
Since $v_k<1-\eta$ in $C_k^j$ one finds for $k$ large that $f_k(v_k)=\ve^{1/2}_kf(v_k)$, so that for $j=1,\dots,N$ it follows
\begin{multline}\label{e15}
F_k(u_k,v_k;C_k^j)\geq \int_{C_k^j}{\Big(\ve_kf^2(v_k)|u'_k|^2+\frac{(1-v_k)^2}{4\ve_k}+\ve_k|v'_k|^2\Big) dx}\\
\geq g^{(\eta)}\left(\left|\int_{C_k^j}u'_k dx\right|\right)\geq g\left(\left|\int_{C_k^j}u'_k dx\right|\right)-\eta^2,
\end{multline}
by Cauchy's inequality and Proposition \ref{p1}.

Outside  the selected components $C_k^j$, $j=1,\dots,N$, one has $v_k\geq 1-\delta$, so that estimate \eqref{6} also holds with $I\setminus \bigcup_{j=1}^N C_k^j$ 
replacing $\Omega\setminus S_{\eta}$. Therefore
\ba{\label{9}
F_k\left(u_k,v_k;I\setminus \bigcup_{j=1}^N C_k^j\right)&\geq&(1-\omega(\delta))\int_{I\setminus \bigcup_{j=1}^N C_k^j}h(|u'_k|)dx\nonumber\\
&\geq& (1-\omega(\delta))\ell\int_{I\setminus \bigcup_{j=1}^N C_k^j}|u'_k|dx-(1-\omega(\delta))\frac{\ell^2}{4}\Le(I\setminus \bigcup_{j=1}^N C_k^j)\nonumber\\
&\geq& (1-\omega(\delta))g\Big(\Big|\int_{I\setminus \bigcup_{j=1}^N C_k^j}u'_k dx\Big|\Big)-\frac{\ell^2}{2}\eta,}
where we have used the definition of $h$ and Proposition \ref{11} \ref{e19}. 

By \eqref{e15}, \eqref{9}, and the subadditivity of $g$ one finds
$$F_k(u_k,v_k;I)+\frac{\ell^2}{2}\eta+\frac{c\,\eta^2}{\delta^2-\eta^2}\geq (1-\omega(\delta))g\left(\left|\int_I u'_k dx\right|\right)
=(1-\omega(\delta))g(|u_k(x_1)-u_k(x_2)|).$$ 

By property \eqref{13} and by the continuity of $g$, as $k\to+\infty$ one deduces 
$$\li F_k(u_k,v_k;I^i_{\eta})+\frac{\ell^2}{2}\eta+\frac{c\eta^2}{\delta^2-\eta^2}\geq (1-\omega(\delta))g(|u(x_1)-u(x_2)|).$$
Finally property \eqref{7} concludes the proof of \eqref{8} as $\mu\to0$.

The thesis follows by summing \eqref{e16} and \eqref{8} for $i=1,\dots,L$ and taking first $\eta\to0$ and finally $\delta\to0$.
\end{proof}

\begin{proposition}[Upper bound]\label{p:limsupunidim}
For all $u\in BV(\Omega)$ there exists $(u_k,v_k)\to(u,1)$ in $\Lr$ such that
$$\limsup_{k\to+\infty}F_k(u_k,v_k)\leq \Phi(u).$$
\end{proposition}
\begin{proof}
Let us consider first the case when $u\in SBV^2(\Omega)$.
By a localization argument it is not restrictive to assume that $J_{u}=\{x_0\}$ and to take $x_0=0$.
We also assume for a while that $u$ is constant in a neighborhood on both sides of $0$. 

With fixed $\eta>0$, we consider $T_\eta>0$ and 
 $\alpha_\eta,\beta_\eta\in H^1\big((0,T_{\eta})\big)$ such that 
$\alpha_\eta(0)=u^-(0),$ $\alpha_\eta(T_{\eta})=u^+(0)$, $0\leq\beta_\eta\leq 1$, $\beta_\eta(0)=\beta_\eta(T_\eta)=1$, and
\be{\label{e17}g\big(|[u](0)|\big)+\eta>\int_0^{T_{\eta}} \left(f^2(\beta_\eta)|\alpha_\eta'|^2+\frac{|1-\beta_\eta|^2}{4}+|\beta_\eta'|^2\right)dt.}
This choice is possible in view of Proposition \ref{p2}, up to a translation of the variable $\alpha_\eta$. 

Let us define $A_k:=(-\frac{\ve_kT_\eta}{2},\frac{\ve_kT_\eta}{2})$ and
\ba
{u_k(x):=\left\{ \begin{array}{ll} 
\ {\displaystyle \alpha_\eta\left(\frac{x}{\ve_k}+\frac{T_\eta}{2}\right)} 
& \textrm{if $x\in A_k$,}\\
u & \text{otherwise},
\end{array} \right.\nonumber\\
v_k(x):=\left\{ \begin{array}{ll} 
\ {\displaystyle \beta_\eta\left(\frac{x}{\ve_k}+\frac{T_\eta}{2}\right)} 
& \textrm{if $x\in A_k$,}\\
1 & \text{otherwise}.
\end{array} \right.
}
An easy computation shows that $(u_k,v_k)\to(u,1)$ in $\Lr$, that $u_k,v_k\in H^1(\Omega)$ for $k$ large, and that for the same $k$
$$F_k(u_k,v_k,\Omega\setminus A_k)\leq\int_{\Omega}|u'|^2dx,$$
being $f_k\leq 1$. Moreover using that $f_k\leq \ve^{1/2}_kf$ and changing the variable $x$ with $y=\frac{x}{\ve_k}+\frac{T_\eta}{2}$ one has
$$F_k(u_k,v_k,A_k)\leq g\big(|[u](0)|\big)+\eta,$$
where we have used \eqref{e17}. Therefore we find
$$F''(u,1)\leq \int_{\Omega}|u'|^2dx+\int_{J_u}(g(|[u]|) +\eta)d\Ho,$$
and then \be{\label{e18} F''(u,1)\leq \Phi(u),}
since $\eta$ is arbitrary. 

Let us remove now the hypothesis that $u$ is constant near $0$. For a function $u\in SBV^2(\Omega)$ with $J_u=\{0\}$, one can consider
the sequence $u_j:=u$ in $\Omega\setminus (-1/j,1/j)$, with $u_j:=u(-1/j)$ in $(-1/j,0)$ and $u_j=u(1/j)$ in $(0,1/j)$. Then $u_j\to u$ in
$L^1(\Omega)$ and $|u'_j|\leq |u'|$ $\mathcal{L}^1$-a.e.\ in $\Omega$, so that by the lower semicontinuity of $F''$ and by the absolute continuity of $u$ on both sides of $0$ we conclude
as $j\to+\infty$ that $u$ still satisfies \eqref{e18}. 

The extension of \eqref{e18} to each $u\in SBV^2(\Omega)$ with
$\Ho(J_u)<+\infty$ is immediate and finally \cite[Propositions 3.3-3.5]{bou-bra-but} conclude the proof. 
\end{proof}

\section[Proof in the \texorpdfstring{$n$}{n}-dimensional case]{Proof in the \texorpdfstring{$n$}{n}-dimensional case}\label{s:ndim}
In this section we establish the $\Gamma$-convergence result in the $n$-dimensional setting.
We recover the lower bound estimate by using a slicing technique thus reducing ourselves
to the one-dimensional setting of Proposition~\ref{p:liminfunidim}. Instead, the upper bound 
inequality follows by an abstract approach based on integral representation results 
(cp. Proposition~\ref{p:limsupndim} below).
\begin{proposition}\label{p:liminfndim}
For every $(u,v)\in \Lr$ it holds
$$F(u,v)\leq F'(u,v).$$
\end{proposition}
\begin{proof}
Let us assume first that $u\in L^\infty(\Omega)$. We set $M:=||u||_{L^\infty(\Omega)}$. 
Let $(u_k,v_k)$ be a sequence such that $(u_k,v_k)\to(u,v)$ in $\Lr$ and 
$\sup F_k(u_k,v_k)<+\infty$. Then it is straightforward that $v=1$ $\qo$. 
We are going to show that $u\in BV(\Omega)$ and that 
\be{\label{e34}\Phi(u)\leq \li F_k(u_k,v_k),}
that proves the thesis under the assumption of the boundedness of $u$.

Given $\xi\in\Sn$, we consider a subsequence $(u_r,v_r)$ of $(u_k,v_k)$ satisfying
$$((u_r)_y^\xi,(v_r)_y^\xi)\to (u_y^\xi,1) \text{ in }
L^1(\Omega_y^\xi){\times} L^1(\Omega_y^\xi) \text{ for } \mathcal{H}^{n-1}\text{-a.e.\ } y\in\Pi^\xi$$
and realizing the lower limit in \eqref{e34} as a limit.

By Fubini's theorem and Fatou's lemma one deduces that
\be{\label{e30}
\liminf_{r\to \infty}\int_{\Omega_y^\xi}\bigg(f_r((v_r)_y^\xi)\left|\nabla ((u_r)_y^\xi)\right|^2+
\frac{(1-(v_r)_y^\xi)^2}{4\ve_r}+\ve_r|\nabla ((v_r)^\xi_y)|^2\bigg)dt <+\infty}
holds for $\mathcal{H}^{n-1}\text{-a.e. } y\in  \Omega^{\xi}$

The one-dimensional result Proposition \ref{p:liminfunidim} yields now that $u^\xi_y\in BV(\Omega^\xi_y)$ and that
\bm{\label{e29}\int_{\Omega^\xi_y} h(|\nabla (u^\xi_y)|)dt+\int_{J_{u^\xi_y}}g(|[u^\xi_y]|)d\Ho+\ell|D^cu^\xi_y|(\Omega^\xi_y)\leq\\
\leq\liminf_{r\to \infty}\int_{\Omega_y^\xi}\bigg(f_r((v_r)_y^\xi)\left|\nabla ((u_r)_y^\xi)\right|^2+
\frac{(1-(v_r)_y^\xi)^2}{4\ve_r}+\ve_r|\nabla ((v_r)^\xi_y)|^2\bigg)dt.}

Let us check that \eqref{e29} implies $u\in BV(\Omega)$ by estimating $\int_{\Omega^\xi}|D(u^\xi_y)|(\Omega^\xi_y)d\Hn$. 
We first notice that 
\be{\label{e32}\int_{\Omega^\xi_y} |\nabla (u^\xi_y)|dt\leq \frac{1}{\ell}\int_{\Omega^\xi_y} h(|\nabla (u^\xi_y)|)dt+\frac{\ell}{4}\Le(\Omega^\xi_y),}
being $h(s)\geq \ell s-{\ell}^2/4$. 

Since
$g(s)/s\to\ell$ as $s\to0$, with fixed $\eta>0$ one has 
\be{\label{e33}g(s)>(\ell-\eta)s\qquad \textrm{for }s<\delta,}
for some $\delta$ sufficiently small.

Therefore \eqref{e32}, \eqref{e33}, and the boundedness of $u$ entail

\bm{|D(u^\xi_y)|(\Omega^\xi_y)\leq\frac{1}{\ell}\int_{\Omega^\xi_y} h(|\nabla (u^\xi_y)|)dt+\frac{\ell}{4}\diam{\Omega}+\frac{1}{\ell-\eta}\int_{\{t\in J_{u^\xi_y}:|[u^\xi_y]|<\delta\}}g(|[u^\xi_y]|)d\Ho\\
+\frac{2M}{g(\delta)}\int_{\{t\in J_{u^\xi_y}:|[u^\xi_y]|\geq\delta\}}g(|[u^\xi_y]|)d\Ho
+|D^cu^\xi_y|(\Omega^\xi_y)\nonumber\\
\leq c+c\left(\int_{\Omega^\xi_y} h(|\nabla (u^\xi_y)|)dt+\int_{J_{u^\xi_y}}g(|[u^\xi_y]|)d\Ho+\ell|D^cu^\xi_y|(\Omega^\xi_y)\right),
}
where $\diam\Omega$ denotes the diameter of $\Omega$ and $c:=\max\{\frac{1}{\ell},\frac{\ell}{4}\diam{\Omega},\frac{1}{\ell-\eta},\frac{2M}{g(\delta)}\}$.
Integrating the last inequality on $\Omega^\xi$ one deduces by \eqref{e29} 
$$\int_{\Omega^\xi}|D(u^\xi_y)|(\Omega^\xi_y)d\Hn\leq c\Hn(\Omega^\xi)+c\sup_k F_k(u_k,v_k).$$
Taking $\xi=e_1,\dots,e_n$ one obtains $u\in BV(\Omega)$.
  
Let us prove now formula \eqref{e34} using  localization. The integration on $\Omega^\xi$ of the one-dimensional estimate 
in \eqref{e29} gives 
\be{\label{e35}\int_\Omega h(|\nabla u\cdot\xi|)dx+\int_{J_{u}}|\nu_u\cdot\xi|g(|[u]|)d\Hn+\ell\int_\Omega |\gamma_u\cdot\xi|d|D^cu|\leq  
\li F_k(u_k,v_k;\Omega),
}
where $\gamma_u:=\frac{d D^c u}{d|D^c u|}$ denotes the density of $D^cu$ with respect to $|D^c u|$. 
Let $E\subset\Omega$ be a Borel set such that $D^a u(E)=0$ and $D^s u(\Omega\setminus E)=0$, and let
$$\lambda:=\Ln\lfloor{\Omega\setminus E}+\Hn\lfloor{J_u}+|D^cu|\lfloor{E\setminus J_u}.$$
Let us consider a countable dense set $D\subset\Sn$ and the functions 
$$\psi_\xi:=h(|\nabla u\cdot\xi|)\chi_{\Omega\setminus E}+|\nu_u\cdot\xi|g(|[u]|)\chi_{J_u}
+\ell|\gamma_u\cdot\xi|\chi_{E\setminus J_u},\qquad \xi\in D.$$
Then (\ref{e35}) gives $(\psi_\xi\lambda)(A)\le F'(u,1,A)$ for all open sets $A\subset\Omega$. Since 
$F'(u,1,\cdot)$ is superadditive, this implies
 $((\sup_\xi\psi_\xi)\lambda)(A)\le F'(u,1,A)$  (see  \cite[Lemma 15.2]{braides})
and therefore the conclusion.

In the general case, if $u\in L^1\setminus L^\infty(\Omega)$ one considers $(u_k^M,v_k)$ and $(u^M,v)$, 
where $u^M:=(-M\vee u)\wedge M$ denotes the truncation at level $M\in(0,+\infty)$. Since the functional 
$F_k$ decreases by truncation and $u_k^M\to u^M$ in $L^1(\Omega)$,
we deduce that $u^M\in BV(\Omega)$ and 
\be{\label{e37}\Phi(u^M)\leq\li F_k(u_k^M,v_k)\leq \li F_k(u_k,v_k).} 
Therefore $u\in GBV(\Omega)$ and \eqref{e34} follows easily from \eqref{e37} as $M\to+\infty$. 
\end{proof}

To prove the limsup inequality we follow an abstract approach. We first show that the
$\overline{\Gamma}$-limit is a Borel measure. The
only relevant property to be checked is the weak subadditivity of the $\Gamma$-limsup. 
This is a consequence of De Giorgi's slicing and averaging argument as shown in the following lemma.
\begin{lemma}\label{l:wsub}
Let $(u,v)\in\Lr$, let $A',A,B\in\mathcal{A}(\Omega)$ with $A'\subset\subset A$, then 
\be{\label{e:subadd}
F''(u,1;A'\cup B)\leq F''(u,1;A)+F''(u,1;B). 
}
\end{lemma}
\begin{proof}
We assume that the right-hand side of \eqref{e:subadd} is finite, so that $u\in GBV(A\cup B)$ and $v=1$ $\Ln$-a.e.\ in $A\cup B$.
We can reduce the problem to the case of functions $u\in BV\cap L^\infty(A\cup B)$. This is 
a straightforward consequence of the fact that the energies $F_k$'s, and thus the $\Gamma$-limsup $F''$, 
are decreasing by truncations. Actually, thanks to $L^1$ lower semicontinuity, they are continuous 
under such an operation.

Under this assumption, let $(u_k^A,v_k^A)$, $(u_k^B,v_k^B)$ be recovery sequences for $(u,1)$ on $A$ 
and $B$ respectively, that is:
 \begin{equation}\label{e:rec1}
  (u_k^A,v_k^A), (u_k^B,v_k^B)\to (u,1) \text{ in } L^1(\Omega){\times} L^1(\Omega),
 \end{equation}
and
\begin{equation}\label{e:rec2}
 \ls F_k(u_k^A,v_k^A;A)=F''(u,1;A),\quad  \ls F_k(u_k^B,v_k^B;B)=F''(u,1;B). 
\end{equation}
Note that, again up to truncations, we may assume that 
\begin{equation}\label{e:rec3}
(u_k^A,v_k^A),\, (u_k^B,v_k^B)\quad\text{are bounded in } L^\infty(\Omega). 
\end{equation}
To simplify the calculations below we introduce the functionals 
$G_k:L^1(\Omega)\times\mathcal{A}(\Omega)\to[0,+\infty]$ given by
\[
 G_k(v;O):=\int_O\left(\frac{(1-v)^2}{4\ve_k}+\ve_k|\nabla v|^2\right)dx,\quad 
 \text{if $v\in H^1(\Omega)$,}
\]
$+\infty$ otherwise.
Notice that
\[
 F_k(u,v;O)=\int_O f_k^2(v)|\nabla u|^2dx+G_k(v;O).
\]
Let $\delta:=\dist(A^\prime,\partial A)>0$, and with fixed $M\in\N$, we set for all 
$i\in\{1,\ldots,M\}$
\[
A_i:=\left\{x\in \Omega:\,\dist(x,A^\prime)<\frac{\delta}{M}i\right\},
\]
and $A_0:=A^\prime$. Clearly, we have $A_{i-1}\subset\subset A_i\subset A$.
Denote by $\varphi_i\in C_c^1(\Omega)$ a cut-off function between $A_{i-1}$ and $A_i$,
i.e., $\varphi_i|_{A_{i-1}}=1$, $\varphi_i|_{A_{i}^c}=0$, and 
$\|\nabla \varphi_i\|_{L^\infty(\Omega)}\leq \frac{2M}{\delta}$. Then, set
\begin{equation}\label{e:uki}
 u_k^i:=\varphi_i\,u_k^A+(1-\varphi_i)u_k^B,
\end{equation}
and
\begin{equation}\label{e:vki}
 v_k^i:=\begin{cases}
 \varphi_{i-1}\,v_k^A+(1-\varphi_{i-1})(v_k^A\wedge v_k^B) & \text{ on } A_{i-1}\cr
         v_k^A\wedge v_k^B & \text{ on } A_i\setminus {A}_{i-1} \cr
         \varphi_{i+1}(v_k^A\wedge v_k^B)+(1-\varphi_{i+1})\,v_k^B & \text{ on } \Omega\setminus {A}_i.
        \end{cases}
\end{equation}
With fixed $i\in\{2,\ldots,M-1\}$, by the very definitions in \eqref{e:uki} and \eqref{e:vki} above $(u_k^i,v_k^i)\in H^1(\Omega){\times}H^1(\Omega)$ and
the related energy $F_k$ on $A'\cup B$ can be estimated as follows
\begin{equation}\label{e:split}
F_k(u_k^i,v_k^i;A'\cup B)\leq F_k(u_k^A,v_k^A;A_{i-2})+F_k(u_k^B,v_k^B;B\setminus {A}_{i+1})+
F_k(u_k^i,v_k^i;B\cap(A_{i+1}\setminus {A}_{i-2})). 
\end{equation}
Therefore, we need to bound only the last term. To this aim we further split the contributions in 
each layer; in estimating each of such terms we shall repeatedly use the monotonicity of $f_k$ and the 
fact that it is bounded by $1$. In addition, a  
positive constant, which may vary from line to line, 
will appear in the formulas below. Elementary computations and the very definitions in \eqref{e:uki} and
\eqref{e:vki} give, using $v^i_k\le v^A_k$,
\begin{multline}\label{e:i-1i-2}
F_k(u_k^i,v_k^i;B\cap(A_{i-1}\setminus {A}_{i-2}))\leq
\int_{B\cap(A_{i-1}\setminus {A}_{i-2})}f_k^2(v_k^A)|\nabla u_k^A|^2\,dx
+G_k(v_k^i;B\cap(A_{i-1}\setminus {A}_{i-2}))\\
\leq c\,\Big(F_k(u_k^A,v_k^A;B\cap(A_{i-1}\setminus {A}_{i-2}))
+F_k(u_k^B,v_k^B;B\cap(A_{i-1}\setminus {A}_{i-2}))\Big)\\
+\frac{c\,M^2\ve_k}{\delta^2}
\int_{B\cap(A_{i-1}\setminus {A}_{i-2})}|v_k^A-v_k^B|^2\,dx,
\end{multline}
\begin{multline}\label{e:ii-1}
F_k(u_k^i,v_k^i;B\cap(A_{i}\setminus {A}_{i-1}))\\
\leq c\int_{B\cap(A_{i}\setminus {A}_{i-1})}f_k^2(v_k^A\wedge v_k^B)
\left(|\nabla u_k^A|^2+|\nabla u_k^B|^2+\frac{4M^2}{\delta^2}|u_k^A-u_k^B|^2\right)\,dx
+G_k(v_k^A\wedge v_k^B;B\cap(A_{i}\setminus {A}_{i-1}))\\
\leq c\Big(F_k(u_k^A,v_k^A;B\cap(A_{i}\setminus {A}_{i-1}))
+F_k(u_k^B,v_k^B;B\cap(A_{i}\setminus {A}_{i-1}))\Big)
+\frac{c\,M^2}{\delta^2}\int_{B\cap(A_{i}\setminus {A}_{i-1})}|u_k^A-u_k^B|^2\,dx,
\end{multline}
and
\begin{multline}\label{e:i+1i}
F_k(u_k^i,v_k^i;B\cap(A_{i+1}\setminus {A}_{i}))\leq
\int_{B\cap(A_{i+1}\setminus {A}_{i})}f_k^2(v_k^B)|\nabla u_k^B|^2\,dx
+G_k(v_k^i;B\cap(A_{i+1}\setminus {A}_{i}))\\
\leq c\,\Big(F_k(u_k^A,v_k^A;B\cap(A_{i+1}\setminus {A}_{i}))
+F_k(u_k^B,v_k^B;B\cap(A_{i+1}\setminus {A}_{i}))\Big)+\frac{c\,M^2\ve_k}{\delta^2}
\int_{B\cap(A_{i+1}\setminus {A}_{i})}|v_k^A-v_k^B|^2\,dx.
\end{multline}
By adding \eqref{e:split}-\eqref{e:i+1i}, we deduce that
\begin{multline*}
F_k(u_k^i,v_k^i;A'\cup B)\leq F_k(u_k^A,v_k^A;A)+F_k(u_k^B,v_k^B;B)\\
+c\,\Big(F_k(u_k^A,v_k^A;B\cap(A_{i+1}\setminus {A}_{i-2}))
+F_k(u_k^B,v_k^B;B\cap(A_{i+1}\setminus {A}_{i-2}))\Big)\\
+\frac{c\,M^2}{\delta^2}\int_{B\cap(A_{i+1}\setminus {A}_{i-2})}|u_k^A-u_k^B|^2\,dx
+\frac{c\,M^2\ve_k}{\delta^2}\int_{B\cap(A_{i+1}\setminus {A}_{i-2})}|v_k^A-v_k^B|^2\,dx.
\end{multline*}
Hence, by summing up on $i\in\{2,\ldots,M-1\}$ and taking the average, for each $k$ we may find an index 
$i_k$ in that range such that
\begin{multline*}
F_k(u_k^{i_k},v_k^{i_k};A'\cup B)\leq F_k(u_k^A,v_k^A;A)+F_k(u_k^B,v_k^B;B)\\
+\frac{c}{M}\,\Big(F_k(u_k^A,v_k^A;B\cap(A\setminus {A^\prime}))
+F_k(u_k^B,v_k^B;B\cap(A\setminus {A^\prime}))\Big)\\
+\frac{c\,M}{\delta^2}\int_{B\cap(A\setminus {A^\prime})}|u_k^A-u_k^B|^2\,dx
+\frac{c\,M\ve_k}{\delta^2}\int_{B\cap(A\setminus {A^\prime})}|v_k^A-v_k^B|^2\,dx.
\end{multline*}
By \eqref{e:rec1} we deduce that $(u_k^{i_k},v_k^{i_k})\to (u,1)$ in $L^1(\Omega){\times} L^1(\Omega)$,
and actually in $L^q(\Omega){\times} L^q(\Omega)$ for all $q\in[1,+\infty)$ thanks to the uniform boundedness
assumption in \eqref{e:rec3}. Therefore, in view of \eqref{e:rec2} and the definition of $\Gamma$-limsup 
we infer that 
\[
F''(u,1;A'\cup B)\leq \left(1+\frac{c}{M}\right)\Big(F''(u,1;A)+F''(u,1;B)\Big).
\]
The conclusion then follows by passing to the limit on $M\uparrow\infty$.
\end{proof}
We next prove that $F''(u,1;\cdot)$ is controlled in terms of the Mumford-Shah functional $\MS$, whose 
definition is given in \eqref{e:MS}. This result gives a first rough estimate for the upper bound inequality. 
We shall improve on the jump part in Proposition~\ref{p:limsupndim} below and finally we shall conclude the proof of the $\Gamma$-limsup
inequality using a relaxation argument.
\begin{lemma}\label{l:boundglimsup} 
For all $u\in L^1(\Omega)$ and $A\in\mathcal{A}(\Omega)$ it holds
\begin{equation}\label{e:FH}
F''(u,1;A)\leq \MS(u;A).
\end{equation}
\end{lemma}
\begin{proof}
Denote by $\psi:[0,1]\to[0,1]$ any nondecreasing 
lower-semicontinuous function such that $\psi^{-1}(0)=0$, $\psi(1)=1$ and
\[
\sup_{k}f_k(s)\leq\psi(s)\qquad\text{for all $s\in[0,1]$},
\]
for instance $\psi=\chi_{(0,1]}$ satisfies all the conditions written above.
Consider the corresponding functionals $AT_k^\psi:L^1(\Omega)\times L^1(\Omega)\to[0,+\infty]$ 
defined in \eqref{e:ATk},
and note that $F_k\leq AT_k^\psi$ for every $k$.
The upper bound inequality for $(F_k)$ then follows at once from the classical 
results by Ambrosio and Tortorelli (cp. \cite{amb-tort2}, and see also \cite{focardi}).
\end{proof}

We are now ready to prove the upper bound inequality.
\begin{proposition}\label{p:limsupndim}
For every $(u,v)\in \Lr$ it holds
\[
F''(u,v)\leq F(u,v).
\]
\end{proposition}
\begin{proof}
Since $L^1$ is separable,  given any subsequence $(F_{k_j})$ of $(F_k)$ we may extract a further 
subsequence, not relabeled for convenience, $\overline{\Gamma}$-converging to some 
$\widehat{F}$ (see \cite[Theorem~16.9]{dalmaso}). 

The functional $\widehat{F}(u,v;\cdot)$ is by definition increasing and inner regular. Since $F_k(u,v;\cdot)$ is 
additive, one easily deduces that $F'$ is superadditive
and from this that its inner regular envelope $\widehat F=(F')_-$ is superadditive (see \cite[Proposition~14.18 or Proposition~16.12]{dalmaso}).

Using Lemma~\ref{l:wsub} one can show that 
$\widehat F=(F'')_-$ is subadditive  (see \cite[Lemma~14.20 and the proof of Proposition~18.4]{dalmaso}).
Therefore $\widehat F$ is the restriction to open sets of the Borel measure
\begin{equation*}
  F_*(u,v;E)=\inf \{ \widehat F(u,v;A): A\in \mathcal{A}(\Omega); E\subset A\}\,,
\end{equation*}
see \cite[Theorem~14.23]{dalmaso}, in the following we identify $\widehat F$ and $F_*$.

If $u\in L^1(\Omega)$ is such that $\MS(u;\Omega)<+\infty$, then by Lemma~\ref{l:boundglimsup} we obtain
$F''(u,1;\cdot)\le \MS(u;\cdot)<+\infty$ on all open sets, and by the regularity 
properties of Radon measures $F''$ coincides with its inner envelope.
Indeed, for a given open set $A$ and $\ve>0$, choose open sets $A'$, $A''$ and $C$ with 
$A'\subset\subset A''\subset\subset A$ 
and $A\setminus A'\subset C$ such that $\MS(u;C)\le\ve$. 
Then use Lemmas~\ref{l:wsub} and \ref{l:boundglimsup} to estimate
$F''(u,1;A)\le F''(u,1;A'\cup C)\le F''(u,1;A'')+\MS(u;C)\leq F''(u,1;A'')+\ve$.
In other words,  $\widehat{F}(u,1)$ is the 
$\Gamma$-limit of $F_{k_j}$ for all $u$ such that $\MS(u)<+\infty$.

 For all $u\in SBV^2(\Omega)$ in particular the estimate in Lemma~\ref{l:boundglimsup} implies that
 \begin{equation}\label{e:upbvolume}
 \widehat{F}(u,1;\Omega\setminus J_u)\leq\int_\Omega |\nabla u|^2dx.
 \end{equation}
We provide below for the same $u$ the estimate
\begin{equation}\label{e:upbvsalto}
 \widehat{F}(u,1;J_u)\leq\int_{J_u}g(|[u]|)\,d\Hn.
\end{equation}
Given this for granted we conclude as follows: we consider the functional $F_\infty:BV(\Omega)\to[0,+\infty]$
\[
F_\infty(u):=\begin{cases}
\displaystyle{
\int_\Omega |\nabla u|^2dx+\int_{J_u}g(|[u]|)\,d\Hn} & \textrm{if $u\in SBV^2(\Omega)$}
\cr
+\infty & \textrm{otherwise on $BV(\Omega)$.}
             \end{cases}
\] 
Further, note that by \cite[Theorem~3.1 and Propositions~3.3-3.5]{bou-bra-but} its relaxation w.r.to the $w\ast\hbox{-}BV$ 
topology is given on $BV(\Omega)$ by $F(\cdot,1)$. 
Since, by \eqref{e:upbvolume} and \eqref{e:upbvsalto} we have that $\widehat{F}\leq F_\infty$, 
and $\widehat{F}(\cdot,1)$ is $L^1$-lower semicontinuous, we infer that 
\[
\widehat{F}(u,1)\leq F(u,1)\quad\textrm{for all $u\in BV(\Omega)$.}
\]
We conclude that the same inequality is true for all $u\in GBV\cap L^1(\Omega)$ by the usual truncation 
argument. 

Finally, by combining the latter estimate with the lower estimate of Proposition~\ref{p:liminfndim}
allows us to deduce that the $\Gamma\hbox{-}$limit does not depend on the chosen subsequence
and it is equal to $F$. Hence, by Urysohn's property the whole family $(F_k)$ $\Gamma\hbox{-}$converges 
to $F$ (cp. \cite[Proposition~8.3]{dalmaso}). 

Let us now prove formula \eqref{e:upbvsalto}. To this aim, fixed $\lambda>0$ we introduce the perturbed 
functional 
\[
\widehat{F_{\lambda}}(u,1):=\widehat{F}(u,1)+\lambda\Big(\int_{\Omega}|\nabla u|^2dx
+\int_{J_u}(1+|[u]|)d\Hn\Big)
\]
for all $u\in SBV^2(\Omega)$. We may apply to $\widehat{F_{\lambda}}$ the integral representation result 
\cite[Theorem~1]{bou-fon-leo-masc} to infer that for $\Hn$-a.e. $x\in J_u$
\bm{\label{e:flambda}
\frac{d\widehat{F_{\lambda}}(u,1;\cdot)}{d(\Hn\res J_u)}(x)=\limsup_{\delta\downarrow 0}\frac{1}{\delta^{n-1}}
\inf\left\{\widehat{F_{\lambda}}(w,1;x+\delta\,Q_{\nu_u(x)}):\,w\in SBV^2\big(x+\delta\,Q_{\nu_u(x)}\big),\right.
\\ \left.w=u_x \text{ on a neighborhood of }x+\delta\,\partial Q_{\nu_u(x)}\right\},
}
where
\[
u_x(y):=
\begin{cases}
 u^+(x) & \text { if }\langle y-x,\nu_u(x)\rangle>0\cr
 u^-(x) & \text { if }\langle y-x,\nu_u(x)\rangle<0
\end{cases}
\]
and $Q_{\nu_u(x)}$ denotes any cube of side $1$ centered in the origin and with a face
orthogonal to $\nu_u(x)$.
Hence, it is enough to show that for $\Hn$-a.e. $x\in J_u$
\begin{equation}\label{e:minest}
\limsup_{\delta\downarrow 0}\frac 1{\delta^{n-1}}\widehat{F}(u_x,1;x+\delta\, Q_{\nu_u(x)})\leq g(|[u](x)|),
\end{equation}
since by taking $u_x$ itself as test function in \eqref{e:flambda} we get
$$
\frac{d\widehat{F_{\lambda}}(u,1;\cdot)}{d(\Hn\res J_u)}(x)\leq \limsup_{\delta\downarrow 0}\frac{1}{\delta^{n-1}}
\widehat{F}(u_x,1;x+\delta\, Q_{\nu_u(x)})+\lambda(1+|[u](x)|),
$$
in turn implying 
\[
 \widehat{F}(u,1;J_u)\leq\widehat{F_{\lambda}}(u,1;J_u)\leq
 \int_{J_u}\big(g(|[u](x)|)+\lambda+\lambda|[u](x)|\big)\,d\Hn.
\]
Finally, \eqref{e:upbvsalto} follows at once by letting $\lambda\downarrow 0$.

Formula \eqref{e:minest} easily follows by repeating the one-dimensional construction 
of Proposition~\ref{p:limsupunidim}. More precisely, assume $x=0$ and $\nu_u(x)=e_n$ for simplicity. 
With fixed $\eta>0$, let $T_{\eta}>0$ and $\alpha_{\eta},\beta_{\eta}\in H^1\big((0,T_{\eta})\big)$ be such that 
$\alpha_{\eta}(0)=u^-(0),$ $\alpha_{\eta}(T_{\eta})=u^+(0)$, $\beta_{\eta}(0)=\beta_{\eta}(T_\eta)=1$, 
$u^-(0)\leq\alpha_{\eta}\leq u^+(0)$, $0\leq\beta_{\eta}\leq 1$, and
\[
\int_0^{T_{\eta}} \left(f^2(\beta_{\eta})|\alpha_{\eta}'|^2+
\frac{|1-\beta_{\eta}|^2}{4}+|\beta_{\eta}'|^2\right)dt\leq g\big(|[u](0)|\big)+\eta.
\]
Let $A_j:=(-\frac{\ve_{k_j}T_\eta}{2},\frac{\ve_{k_j}T_\eta}{2})$, and set
\ba
{u_j(y):=\left\{ \begin{array}{ll} 
\ {\displaystyle \alpha_{\eta}\left(\frac{y_n}{\ve_{k_j}}+\frac{T_\eta}{2}\right)} 
& \textrm{if $y_n\in A_j$}\\
u_0 & \text{otherwise},
\end{array} \right.\nonumber\\
v_j(y):=\left\{ \begin{array}{ll} 
\ {\displaystyle \beta_{\eta}\left(\frac{y_n}{\ve_{k_j}}+\frac{T_\eta}{2}\right)} 
& \textrm{if $y_n\in A_j$}\\
1 & \text{otherwise}.
\end{array} \right.
}
Clearly, $(u_j,v_j)\to(u_0,1)$ in $L^1(Q_{e_n})\times L^1(Q_{e_n})$, and if 
$Q'_{e_n}=Q_{e_n}\cap(\R^{n-1}\times\{0\})$, a change of variable yields 
\begin{multline*}
F_{k_j}(u_j,v_j;\delta\,Q_{e_n})=F_{k_j}\big(u_j,v_j,\delta\,Q'_{e_n}\times A_j\big)\\
\leq\delta^{n-1}\int_0^{T_{\eta}} \left(f^2(\beta_{\eta})|\alpha_{\eta}'|^2+
\frac{|1-\beta_{\eta}|^2}{4}+|\beta_{\eta}'|^2\right)dt\leq \delta^{n-1}(g\big(|[u](0)|\big)+\eta).
\end{multline*}
Therefore, by the very definition of $\widehat{F}$ we infer that
\[
\widehat{F}(u_0,1;\delta\,Q_{e_n})\leq \delta^{n-1}(g\big(|[u](0)|\big)+\eta),
\]
and estimate \eqref{e:minest} follows at once dividing by $\delta^{n-1}$ and taking the superior 
limit as $\delta\downarrow 0$, and finally by letting $\eta\downarrow 0$ in the formula above.
 \end{proof}

The proof of the compactness result Theorem~\ref{t:comp} follows the lines of \cite[Theorem 7.4]{dm-iur},
so we just sketch the relevant arguments and refer to \cite{dm-iur} for more details.
\begin{proof}[Proof of Theorem~\ref{t:comp}]
One first proves the thesis in the one-dimensional case under the hypothesis that $u_k$ is bounded in 
$L^\infty(\Omega)$. Then one extends the proof to the $n$-dimensional case and finally removes the 
boundedness assumption. 

Let us start assuming that $n=1$ and that $\sup_k ||u_k||_{L^\infty(\Omega)}<+\infty$. 
Up to a diagonalization argument, one reduces to study the case $\Omega=(0,1)$. Repeating the proof of 
Theorem~\ref{p:liminfunidim} one finds that $v_k\to1$ in $L^1(\Omega)$ and that for every $\delta>0$ there 
exists a finite subset $S\subset\Omega$ for which
$$(1-\omega(\delta))\int_{\Omega\setminus S_{\eta}}h(|u'_k|)dx\leq F_k(u_k,v_k,\Omega\setminus S_{\eta})$$
holds for $\eta>0$ small (dependently on $\delta$) and for $k$ large (dependently on $\eta$),
where $\omega$ is a modulus of continuity provided by $f$ and $S_{\eta}:=\bigcup_{i=1}^L(t_i-\eta,t_i+\eta)$.
This implies by assumption that $u_k$ is bounded in $BV(\Omega\setminus S_{\eta})$ uniformly with respect 
to $k$ and $\eta$.
Hence up to subsequences $u_k$ converges to a function $u\in BV(\Omega\setminus S_{\eta})$ $\Le$-a.e.\ in 
$\Omega\setminus S_{\eta}$.
The boundedness hypothesis and a diagonalization argument yield that $u$ in fact belongs to $BV(\Omega)$ 
and that $u_k\to u$ $L^1(\Omega)$. 

\medskip
In order to generalize the previous result to the case $n>1$, one applies a compactness result by 
Alberti, Bouchitt\'{e}, and Seppecher \cite[Theorem 6.6]{alberti}. Indeed, fixed $\xi\in\Sn$ and $\delta>0$, 
one can introduce the sequence $w_k$ whose slices satisfy
\bes{
& (w_k)^\xi_y:=\begin{cases}(u_k)^{\xi}_y & \textrm{if $y\in A_k$,}\\
                             0            & \textrm{otherwise,} 
						 \end{cases}\\
& A_k:=\{y\in \Omega^{\xi}:F^1_k((u_k)^{\xi}_y,(v_k)^\xi_y)\leq L\},}
where $F_k^1$ denotes the one-dimensional counterpart of the functional $F_k$ and $L$ is  chosen properly and
depends on $\delta$.
An easy computation shows that $w_k$ is bounded in $L^\infty(\Omega)$, that $u_k$ is in a $\delta$-neighborhood 
of $w_k$ in $L^1(\Omega)$, and that $(w_k)^\xi_y$ is pre-compact in $L^1(\Omega)$ (the last property follows 
from the first part of the proof).
Then the pre-compactness of $u_k$ in $L^1(\Omega)$ is ensured by \cite[Theorem 6.6]{alberti} as $\xi$ varies 
in a basis of $\Rn$.

\medskip
If $u_k$ is not bounded in $L^\infty(\Omega)$ the argument above applies to the truncations, so that up 
to subsequences $u_k^M\to u_M$ in $L^1(\Omega)$ and $\qo$, with $u_M\in BV(\Omega)$, for every $M\in \N$. 
One can prove that the function
$$u:=\lim_{M\to+\infty}u_M$$
is well-defined, finite $\qo$, and its truncation $u^M$ coincides with $u_M$ $\qo$. This straightforwardly 
implies that $u_k\to u$ $\qo$ and that $u\in GBV\cap L^1(\Omega)$.
\end{proof}

\section{Further results}\label{s:further}
In this section we build upon the results in Sections~\ref{s:stat}-\ref{s:ndim} to obtain in the 
limit different models by slightly changing the approximating energies $F_k$'s.
More precisely, we shall approximate a cohesive model with the Dugdale's surface density, a 
cohesive model with power-law growth at small openings, and a model in Griffith's brittle fracture.

This task will be accomplished by letting the function $f$ vary as in item (ii) 
of Proposition~\ref{p:gl} in the first instance, as in item (iii) in the third, and suitably
in the second (cp. (iii) of Proposition~\ref{p:tp} below), respectively. 
More precisely, we consider a sequence of functions $(\f{j})$ satisfying \eqref{f0} and \eqref{f1}
and for all $j,k\in\N$ introduce the energies
\begin{equation}\label{e:Fj}
\Fk jk(u,v):=\begin{cases} 
\displaystyle{\int_\Omega{\Big((\fk jk)^2(v)|\nabla u|^2+\frac{(1-v)^2}{4\ve_k}
+\ve_k|\nabla v|^2\Big) dx}} 
& \textrm{if $(u,v)\in \Hr$}\cr
& \text{and $0\leq v\leq1$ $\qo$},\cr
+ \infty & \text{otherwise},
\end{cases} 
\end{equation}
where $\fk jk(s):=1\wedge\ve_k^{1/2} \f{j}(s)$.

In each of Theorems~\ref{t:Dugdale}, \ref{t:sublin}, and \ref{t:MS} below we shall further
specify the nature of the sequence $(\f{j})$.

\subsection{Dugdale's cohesive model}
\label{subsecdugdale}
In order to approximate the Dugdale's model, i.e., to get in the limit 
$\D:L^1(\Omega)\to[0,+\infty]$ 
\begin{equation}\label{e:Phit}
\D(u):=
  \begin{cases} 
\displaystyle{\int_\Omega h(|\nabla u|)dx+\int_{J_{u}}\big(\gD |[u]|\big)d\Hn+\ell|D^cu|(\Omega)}
& \textrm{if $u\in GBV(\Omega)$,}\cr\cr
+\infty & \textrm{otherwise},
  \end{cases} 
\end{equation}
with $h$ as in \eqref{h},
we shall consider the specific choice
\begin{equation}\label{e:fjPhit}
\f j(s):=(\aj_j\,s)\vee f(s)
\end{equation}
with $f$ satisfying \eqref{f0} and \eqref{f1}, and
\begin{equation}\label{e:alphaj}
\text{$(\aj_j)$ nondecreasing, $\aj_j\uparrow\infty$ and such that 
$\aj_j\,\ve_j^{1/2}\downarrow 0$.}
\end{equation}
\begin{theorem}\label{t:Dugdale}
Suppose that $(\f j)$ is as in \eqref{e:fjPhit} and \eqref{e:alphaj} above.

Then, the functionals $\Fk kk$ $\Gamma$-converge in $\Lr$ to the functional $\Dt$ 
defined as follows  
\begin{equation}\label{e:F1}
\Dt(u,v):=\begin{cases} 
\D(u) & \textrm{if $v=1$ $\qo$,}\cr
+\infty & \text{otherwise}.
\end{cases} 
\end{equation}
\end{theorem}
\begin{proof}[Proof of Theorem~\ref{t:Dugdale}]
The very definitions in \eqref{Fk} and \eqref{e:Fj} give $\Fk jk\leq \Fk kk$ for $j\leq k$, 
being $(\f j)$ nondecreasing by assumption. Hence, by Theorem~\ref{t:gamma-lim} we deduce
\begin{equation}\label{e:liminfFj0}
\Gamma\hbox{-}\liminf_k \Fk kk(u,v)\geq \F j(u,v),
\end{equation}
where $\F j$ is defined as $F$ in \eqref{F} with $f$ substituted by $\f j$ in formulas
\eqref{f1} and \eqref{h} defining the volume density, and \eqref{g} defining the surface density. 

In particular, being $\ell_j=\ell$ for all $j$, the corresponding volume density $h_j$ 
equals the function $h$ in \eqref{h}. 
Moreover, the surface energy densities $g_j$ are dominated by the constant $1$, 
and by item (ii) in Proposition~\ref{p:gl} we have $\lim_jg_j(s)=\gD s$ 
for all $s\in[0,+\infty)$. 
In conclusion, if $\Gamma\hbox{-}\liminf_k \Fk kk(u,v)<+\infty$, we infer that 
$v=1$ $\qo$, $u\in GBV(\Omega)$, and
\[
\Gamma\hbox{-}\liminf_k\Fk kk(u,1)\geq \D(u),
\]
by the Dominated Convergence theorem, as $j\uparrow+\infty$ in \eqref{e:liminfFj0}.

The upper bound inequality follows by arguing as in Proposition~\ref{p:limsupndim}.
Indeed, we first note that by a careful inspection of the proofs, Lemmas~\ref{l:wsub} 
and \ref{l:boundglimsup} are still valid in this generalized framework. 
More precisely, Lemma~\ref{l:wsub} continue to hold true as there we have only used that each 
function $f_k=1\wedge\ve_k^{1/2} f$ in \eqref{fk} is nondecreasing and bounded by $\chi_{(0,1]}$, 
properties enjoyed by $\fk kk$ as well.

In conclusion, as a first step we establish the estimate
\begin{equation}\label{e:minest2}
\limsup_{\delta\downarrow 0}\frac 1{\delta^{n-1}}\widehat{F}(u_x,1;x+\delta\, Q_{\nu_u(x)})\leq 
\gD{|[u](x)|},
\end{equation}
for $u\in SBV^2(\Omega)$ and for $\Hn$-a.e. $x\in J_u$, where $\widehat{F}$ is the 
$\overline{\Gamma}$-limit of a properly chosen subsequence $(\Fk{k_j}{k_j})$ of $(\Fk kk)$
(cp. Proposition~\ref{p:limsupndim}).

Given \eqref{e:minest2}, the derivation of the upper bound inequality in general follows 
exactly as in Proposition~\ref{p:limsupndim}.

Let us now prove \eqref{e:minest2} by means of a one-dimensional construction.
For the sake of simplicity we assume $x=0$ and $\nu_u(x)=e_n$.
Actually, in view of the estimate in Lemma~\ref{l:boundglimsup} we need only to discuss 
the case $|[u](0)|<\ell^{-1}$. To this aim, set 
\[
s_j:=\sup\{s\in[0,1):\,\aj_{k_j}s=f(s)\},                 
\]
it is easy to check that $s_j$ it is actually a maximum, i.e., $\aj_{k_j}s_j=f(s_j)$, and that 
$s_j\leq s_{j+1}<1$ with $s_j\uparrow 1$.
Let now 
\[
T_{j}:=|[u](0)|\frac{f(s_j)}{1-s_j}, 
\]
then $T_j\uparrow\infty$. Define $\alpha_j(t):=u^-(0)$ on $[-T_j-1,-T_j]$, 
$\alpha_j(t):=[u](0)\cdot(\frac{t}{2T_j}+\frac12)+u^-(0)$ on $[-T_j,T_j]$, 
$\alpha_j(t):=u^+(0)$ on $[T_j,T_j+1]$, and $\beta_j(t):=s_j$ on $[-T_j,T_j]$, 
$\beta_j(t):=(1-s_j)(|t|-T_j)+s_j$ otherwise in $[-T_j-1,T_j+1]$.

Setting $A_j:=(-{\ve_{k_j}(T_j+1)},{\ve_{k_j}(T_j}+1))$, we have that $\Le(A_j)\to0$ as $j\uparrow\infty$ 
by \eqref{e:alphaj}. Indeed, in view of \eqref{f1} and the definition of $s_j$ it is easy to deduce
that $(1-s_j)\aj_{k_j}\to\ell$ as $j\uparrow\infty$, so that 
$\ve_{k_j}T_j\sim\ve_{k_j}\aj_{k_j}^2\to 0$ 
as $j\uparrow\infty$ thanks to \eqref{e:alphaj}. Therefore, if 
\ba
{u_j(y):=\left\{ \begin{array}{ll} 
 {\displaystyle \alpha_j\left(\frac{y_n}{\ve_{k_j}}\right)} 
& \textrm{if $y_n\in A_j$}\\
u_0 & \text{otherwise},
\end{array} \right.\nonumber\\
v_j(y):=\left\{ \begin{array}{ll} 
 {\displaystyle \beta_j\left(\frac{y_n}{\ve_{k_j}}\right)} 
& \textrm{if $y_n\in A_j$}\\
1 & \text{otherwise},
\end{array} \right.
}
then $(u_j,v_j)\to (u_0,1)$ on $L^1(Q_{e_n})\times L^1(Q_{e_n})$,
where  $u_0=u^-(0)\chi_{\{y_n\leq 0\}}+u^+(0)\chi_{\{y_n> 0\}}$.

Moreover, if $Q'_{e_n}=Q_{e_n}\cap(\R^{n-1}\times\{0\})$, then a change of variable yields 
\begin{multline*}
\Fk{k_j}{k_j}(u_j,v_j;\delta\,Q_{e_n})=\Fk {k_j}{k_j}\big(u_j,v_j;\delta\,Q'_{e_n}\times A_j\big)\\
\leq\delta^{n-1}\left(\int_{-T_j}^{T_j}\bigg(f^2(\beta_j)|\nabla\alpha_j|^2+\frac{(1-\beta_j)^2}4\bigg)dt
+2\int_{T_j}^{T_j+1} \left(\frac{|1-\beta_j|^2}{4}+|\beta'_j|^2\right)dt\right)\\
=\delta^{n-1}\left(
\bigg(f^2(s_j)\frac{|[u](0)|^2}{2T_j}+2T_j\frac{(1-s_j)^2}4\bigg)
+2(1-s_j)^2\int_{T_j}^{T_j+1}\frac{(t-(T_j+1))^2}{4}dt+2(1-s_j)^2\right)\\
=\delta^{n-1}\left((1-s_j)f(s_j)|[u](0)|+\frac{13}6(1-s_j)^2\right)
=\delta^{n-1}\big(\ell|[u](0)|+o(1)\big)\qquad\textrm{as $j\uparrow\infty$}.
\end{multline*}
Therefore, being $|[u](0)|<\ell^{-1}$, by the very definition of $\widehat{F}$ we infer that
\[
\widehat{F}(u_0,1;\delta\,Q_{e_n})\leq \delta^{n-1}(\gD{|[u](0)|}),
\]
and estimate \eqref{e:minest2} follows at once dividing by $\delta^{n-1}$ and taking the superior 
limit as $\delta\downarrow 0$ in the formula above.
\end{proof}
\begin{remark}\label{r:regimes}
The analysis in the general case of a diverging sequence $f^{(k)}$ is much more intricate because 
of the combination of several effects: the speed of divergence of the $f^{(k)}$'s compared with the 
scaling $\ve_k^{1/2}$ in the definition of $f^{(k)}_k$, and even more the behavior of each $f^{(k)}$ 
close to $1$.
In this remark we limit ourselves to consider those families of functions $f^{(k)}$ satisfying item 
(ii) in Proposition~\ref{p:gl}, another instance shall be discussed in Remark~\ref{r:regimes2} below. 

Therefore, assume for example that $f(s)=\frac{\ell s}{1-s}$, and that 
$\f k$ is defined as in \eqref{e:fjPhit} above but with $a_k=\ve_k^{-1/2}$, thus violating the 
last condition in \eqref{e:alphaj}.
Then, one can show that the $\Gamma$-limit is given by the Mumford-Shah energy introduced 
in \eqref{e:MS}. This claim follows easily by noting that with this choice
\[
\fk kk(s)=\begin{cases}
                    s & 0\leq s\leq 1-\ell\,\ve_k^{1/2} \cr
                    \displaystyle{\ve_k^{1/2}\frac{\ell s}{1-s}} 
                    & 1-\ell\,\ve_k^{1/2}\leq s\leq (1+\ell\,\ve_k^{1/2})^{-1} \cr
                    1 & (1+\ell\,\ve_k^{1/2})^{-1}\leq s\leq 1,
                   \end{cases}
\]
so that $\fk kk(s)\geq s$ for all $s\in[0,1]$, and 
actually $(\fk kk)$ converges uniformly to the identity on $[0,1]$.
Therefore, $AT_k^{Id}\leq\Fk kk\leq AT_k^{\psi}$, with $\psi(s)=\chi_{(0,1]}(s)$
(cp. with \eqref{e:ATk} for the definition of $AT_k^\psi$), 
and the result follows at once from Ambrosio and Tortorelli classical results 
(cp. \cite{amb-tort2}, see also \cite{focardi}).

A similar argument works also in the regime $\aj_k\ve_k^{1/2}\uparrow\infty$, in which
\[
\fk kk(s)=\begin{cases}
                    \aj_k\ve_k^{1/2}s & 0\leq s\leq \aj_k^{-1}\ve_k^{-1/2} \cr
                    1 & \aj_k^{-1}\ve_k^{-1/2}\leq s\leq 1,
                   \end{cases}
\]
for $k$ sufficiently large, so that $\fk kk(s)\to\chi_{(0,1]}(s)$ for all $s\in[0,1]$, and 
again we get the Mumford-Shah energy in the $\Gamma$-limit arguing as above.

Finally, note that for $\aj_k$ as in \eqref{e:alphaj}, we have 
 \[
\fk kk(s)=\begin{cases}
                    \aj_k\ve_k^{1/2}s & 0\leq s\leq 1-\ell\aj_k^{-1} \cr
                    \ve_k^{1/2}\frac{\ell s}{1-s} & 1-\ell\aj_k^{-1}\leq s\leq (1+\ell\ve_k^{1/2})^{-1}\cr
                    1 & (1+\ell\ve_k^{1/2})^{-1}\leq s\leq 1
                   \end{cases}
\]
so that $\fk kk(s)\to\chi_{\{1\}}(s)$ in $[0,1]$.
\end{remark}
\medskip

\subsection{A model with power-law growth at small openings}
\label{subsectpowerlaw}

In Theorem~\ref{t:sublin} below we approximate a model with sublinear surface density 
in the origin and quadratic growth for the volume term. To this aim, let $p>1$ and
consider a function $\psi_p$ satisfying condition \eqref{f0} and 
\be{\label{e:f1p}
\displaystyle\lim_{s\to1}(1-s)^p\psi_p(s)=\kappa, \quad \kappa\in(0,+\infty).
}
Clearly, one can take $\psi_p(s):=\frac{s}{(1-s)^p}$ as prototype. 
The surface energy density $\vartheta_p:[0,+\infty)\to[0,+\infty)$ is defined as $g$ in \eqref{g} 
by 
\be{\label{e:tp}
\vartheta_p(s):=\inf_{(\alpha,\beta)\in\mathcal{U}_s}
\int_0^1|1-\beta|\sqrt{\psi_p^2(\beta)|\alpha'|^2+|\beta'|^2}\,dt,
}
where $\mathcal{U}_s$ has been introduced in \eqref{e:admfnctns}. In this case the integral is 
finite only if $\beta<1$ almost everywhere on the set $\{\alpha'\ne0\}$.
We next prove some properties of $\vartheta_p$ in analogy to 
Propositions~\ref{11}, \ref{p2} and \ref{p:gl}. In what follows, we keep the same notations 
introduced there. We also note that given any curve $(\alpha,\beta)$, the integral to be minimized 
in the definition of $\vartheta_p$ is invariant under reparametrizations of $(\alpha,\beta)$. 
\begin{proposition}\label{p:tp}
 Let $\psi_p$ satisfy \eqref{f0} and \eqref{e:f1p}, let $\vartheta_p:[0,+\infty)\to[0,+\infty)$ be 
 the corresponding surface energy in \eqref{e:tp}. Then, 
\begin{itemize}
 \item[(i)] $\vartheta_p(0)=0$, $\vartheta_p$ is nondecreasing, subadditive, and 
 \be{\label{e:tp4}
 0\leq\vartheta_p(s)\leq 1\wedge c\, s^{\frac2{p+1}},\quad\textrm{for all $s\geq 0$},
 }
 where $c=c(\psi_p)>0$. 
  Moreover, $\vartheta_p\in C^{0,\frac 2{p+1}}\big([0,+\infty)\big)$ and
 \begin{equation}\label{e:tp3}
\kappa^{\frac 2{p+1}}\le \lim_{s\downarrow 0}\frac{\vartheta_p(s)}{s^{\frac2{p+1}}}\le 
\frac{p+1}{2^{\frac 2{p+1}}(p-1)^{\frac{p-1}{p+1}}}\,{\kappa}^{\frac{2}{p+1}};
\end{equation}

 \item[(ii)] $\vartheta_p=\hat{\vartheta}_p$, where
\be{\label{e:tp2}
\hat{\vartheta}_p(s):=\lim_{T\uparrow\infty}\inf_{(\alpha,\beta)\in\mathcal{U}_s(0,T)}
\int_0^T\left(\psi_p^2(\beta)|\alpha^\prime|^2+\frac{(1-\beta)^2}{4}+|\beta^\prime|^2\right)dt;
}
\item[(iii)] the functions
\begin{equation}\label{e:fjPhitt}
f^{(j)}(s):=\frac{j\,s}{1-s}\wedge \psi_p(s),
\end{equation}
 satisfy \eqref{f0} and \eqref{f1}. If $g_j$ denotes the corresponding surface energy in 
\eqref{g}, then $g_j\leq g_{j+1}$ and 
 \be{\label{e:gjtp}
  \lim_{j\to\infty}g_j(s)=\vartheta_p(s)\quad\textrm{for all $s\geq 0$}.
}
  \end{itemize}
\end{proposition}
\begin{proof} We prove (i). The facts that $\vartheta_p(0)=0$ and that $\vartheta_p$ is nondecreasing 
follow easily from the definition. The subadditivity follows as in
Proposition~\ref{11}\ref{e21}.
Moreover, $0\leq\vartheta_p\leq 1$ arguing as in \ref{e19} of Proposition~\ref{11}.

To show \eqref{e:tp4} and the upper bound in \eqref{e:tp3}, let $s,\,\lambda>0$ and consider 
$\alpha:=0$ in $[0,1/3]$, $\alpha:=s$ in $[2/3,1]$ and set $\alpha$ to be the linear interpolation of 
the values $0$ and $s$ on $[1/3,2/3]$; $\beta_\lambda:=1-(\lambda\,s)^{\frac 1{p+1}}$ in $[1/3,2/3]$ 
and set $\beta_\lambda$ to be the linear interpolation of that value to $1$ on $[0,1/3]\cup[2/3,1]$.

Then, clearly $(\alpha,\beta_\lambda)\in\mathcal{U}_s$ and a simple computation shows that 
\begin{equation}\label{e:thetapub}
 \vartheta_p(s)\leq\int_0^1|1-\beta_\lambda|\sqrt{\psi_p^2(\beta_\lambda)|\alpha'|^2
 +|\beta_\lambda^{\prime}|^2}\,dt
 =(\lambda\,s)^{\frac{1}{p+1}}\,\psi_p\big(1-(\lambda\,s)^{\frac 1{p+1}}\big)\,s
 +(\lambda\,s)^{\frac 2{p+1}}.
\end{equation}
By taking $\lambda=1$, since $(1-t)^p\psi_p(t)\leq c$ for some constant $c=c(\psi_p)>0$ and 
for all $t\in[0,1]$, we deduce that 
\[
\vartheta_p(s)\leq (c+1)\,s^{\frac 2{p+1}},
\]
from which inequality \eqref{e:tp4} follows as $0\leq\vartheta_p\leq 1$.

Note that the H\"older continuity of $\vartheta_p$ then follows easily from \eqref{e:tp4} and 
its subadditivity and monotonicity.

Further, by \eqref{e:thetapub} we infer 
\[
\limsup_{s\downarrow 0}\frac{\vartheta_p(s)}{s^{\frac2{p+1}}}\le 
\kappa\,\lambda^{-\frac{p-1}{p+1}}+\lambda^{\frac{2}{p+1}},
\]
minimizing the latter inequality over $\lambda\in(0,\infty)$ yields the upper bound in  (\ref{e:tp3}).
\smallskip

We now prove the lower bound in \eqref{e:tp3}. Let $s_k\to 0$, $s_k>0$, and up to subsequences let the liminf in \eqref{e:tp3} be a limit. Let $\alpha_k,\beta_k$ be competitors for $\vartheta_p(s_k)$ such that
$$\int_{0}^{1} |1-\beta_k| \, \sqrt{\psi_p^2(\beta_k)|\alpha_k'|^2 + |\beta_k'|^2}\, dt\le \vartheta_p(s_k)+s_k.$$
If, after taking a  subsequence, there is a sequence $x_j\in[0,1]$ such that
$$\frac{1-\beta_j(x_j)}{s_j^{\frac{1}{p+1}}}\ge\kappa^{1/(p+1)} \text{ for all $j$},$$
then 
\be{\label{e:MM}\vartheta_p(s_j)+s_j\ge
(1-\beta_j(x_j))^2\ge\kappa^{2/(p+1)}\,s_j^{\frac{2}{p+1}}.}
Otherwise, for all $k$ large  enough
$$\frac{1-\beta_k}{s_k^{\frac{1}{p+1}}}\leq \kappa^{1/(p+1)}$$
must hold uniformly, so that $\beta_k\to1$ uniformly and by (\ref{e:f1p}) for any $\varepsilon>0$
\begin{equation*}
  (1-\beta_k)^p \psi_p(\beta_k) \ge \kappa-\varepsilon \text{ uniformly, for $k$ large enough.}
\end{equation*}
Therefore
\begin{equation}\label{e:the}
\vartheta_p(s_k)+s_k>\int_0^1\psi_p(\beta_k)(1-\beta_k)|\alpha'_k|dt\geq \int_0^1 \frac{\psi_p(
\beta_k)(1-\beta_k)^p}{(1-\beta_k)^{p-1}}|\alpha'_k|dt 
\ge \frac{\kappa-\varepsilon}{\kappa^{(p-1)/(p+1)}} s_k^{2/(p+1)}\,.
\end{equation}
Since $\varepsilon$ was arbitrary this and  (\ref{e:MM}) give
the lower bound in   (\ref{e:tp3}).

Finally we prove that the limit in   (\ref{e:tp3}) exists. We fix a sequence $s_j\downarrow 0$ and choose 
$\alpha_j,\beta_j \in\mathcal{U}_{s_j}$ such that
\begin{equation*}
\int_0^1|1-\beta_j|\sqrt{\psi_p^2(\beta_j)|\alpha'_j|^2  +|\beta_j^{\prime}|^2}\,dt
\le  \vartheta_p(s_j) + \frac1j s_j^{2/(p+1)}\,.
  \end{equation*}
By the computation above we obtain $\beta_j\to1$ uniformly. 
For $k\ge j$ we define  $\alpha_k,\beta_k\in\mathcal{U}_{s_k}$ by
\begin{equation*}
  \overline \alpha_k =\frac{s_k}{s_j}\alpha_j \text{ and }
  \overline \beta_k =1- \Bigl(\frac{s_k}{s_j}\Bigr)^{1/(p+1)}(1-\beta_j)\,.
\end{equation*}
After a straightforward computation, using these test functions in 
the definition of $\vartheta_p(s_k)$ leads to
\begin{equation*}
  \vartheta_p(s_k)\le \Bigl(\frac{s_k}{s_j}\Bigr)^{2/(p+1)} \left[
\int_0^1|1-\beta_j|\sqrt{\psi_p^2(\beta_j)|\alpha'_j|^2  +|\beta_j^{\prime}|^2}\,dt\right]
\sup\bigl\{\frac{\psi_p(t)(1-t)^p}{\psi_p(t')(1-t')^p}: \min \beta_j\le t,t'<1\bigr\}\,.
\end{equation*}
Since $\beta_j\to1$ uniformly as $j\to\infty$, and $\psi_p(t)(1-t)^p$ has a finite limit as $t\to1$, the
$\sup$  converges to $1$ as $j\to\infty$. Therefore we obtain that for every $\ve>0$ if $j$ is sufficiently large, then
\begin{equation*}
  \frac{\vartheta_p(s_k)}{s_k^{2/(p+1)}} \le 
  (1+\ve) \frac{\vartheta_p(s_j)}{s_j^{2/(p+1)}} +\frac1j\hskip5mm \text{ for all } k\ge j\,.
\end{equation*}
This implies that the sequence converges. Since the decreasing 
sequence $s_j$ was  arbitrary, the limit in  (\ref{e:tp3}) exists.

\smallskip
To establish (ii), we note first that by Cauchy inequality $\vartheta_p\leq\hat{\vartheta}_p$.
For the sake of proving the converse inequality, we first claim that
$\alpha$ and $\beta$ in the infimum problem defining $\vartheta_p$ can be taken in $W^{1,\infty}\big((0,1)\big)$. Let $\eta>0$ small and let $\alpha,\beta\in H^1\big((1/3,2/3)\big)$
be competitors for $\vartheta_p(s)$ such that
\be{\label{e:quasimin2}\int_{1/3}^{2/3} |1-\beta| \, \sqrt{\psi_p^2(\beta)|\alpha'|^2 + |\beta'|^2}\, dt\le \vartheta_p(s)+\eta.}
We define $\beta^\eta(t):=\beta(t)\wedge (1-\eta)$ in $[1/3,2/3]$. 
Since $(1-s)^p\psi_p(s)$ has a finite nonzero limit 
at $1$, there is a function $\omega$, with $\omega(\eta)\to0$ as $\eta\to0$, such that
\begin{alignat}1\label{al:com}
  (1-s')^p \psi_p(s') &\le (1+  \omega(\eta))  (1-s)^p \psi_p(s) \text{ for all } s,s'\in [1-\eta,1)\,.
\end{alignat}
In particular, if $1-\eta< \beta(t)< 1$,  then
\begin{equation}\label{e:cont}
  \eta \psi_p(1-\eta) \le \eta^{1-p} (1+\omega(\eta)) (1-\beta(t))^p \psi_p(\beta(t)) \le 
 (1+\omega(\eta)) (1-\beta(t)) \psi_p(\beta(t)) \,.
\end{equation}
We observe that $\beta^\eta=1-\eta$ and $(\beta^\eta)'=0$ almost everywhere on the set 
$\{\beta\ne \beta^\eta\}$ and compute
\begin{alignat}1
\int_{\{\beta\ne \beta^\eta\}}
(1-\beta^\eta) \sqrt{\psi_p^2(\beta^\eta)\,  |\alpha'|^2 + |(\beta^\eta)'|^2}dt &=
\int_{\{\beta\ne \beta^\eta\}}\eta \psi_p(1-\eta)\,  |\alpha'| dt \nonumber\\
&\le (1+\omega(\eta))\int_{\{\beta\ne \beta^\eta\}}(1-\beta) \psi_p(\beta) |\alpha'|\,dt,\label{e:mini}
\end{alignat}
so that by \eqref{e:quasimin2} it follows
\be{\label{e:betaeta}\int_{1/3}^{2/3} |1-\beta^\eta| \, \sqrt{\psi_p^2(\beta^\eta)|\alpha'|^2 + |(\beta^\eta)'|^2}\, dt\le \vartheta_p(s)+\eta+\omega(\eta)+\eta\omega(\eta).}
By density we are able to find two sequences $\alpha_j,\beta_j^\eta\in W^{1,\infty}\big((1/3,2/3)\big)$ (actually in $C^\infty([1/3,2/3])$) such that
$\alpha_j(1/3)=0$, $\alpha_j(2/3)=s$, $\beta_j^\eta(1/3)=\beta_j^\eta(2/3)=1-\eta$, $0\leq\beta\leq 1-\eta$, and converging respectively 
to $\alpha$ and $\beta^\eta$ in $H^1\big((1/3,2/3)\big)$.

Since the function $(1-s)^p\psi_p(s)$ is uniformly continuous in $[0,1-\eta]$ and since $\beta^\eta_j\to\beta^\eta$ also uniformly, we deduce that
for $j$ large it holds
$$\int_{1/3}^{2/3} |1-\beta_j^\eta| \, \sqrt{\psi_p^2(\beta_j^\eta)|\alpha_j'|^2 + |(\beta_j^\eta)'|^2}\, dt\le \vartheta_p(s)+2\eta+\omega(\eta)+\eta\omega(\eta).$$
Finally we extend $\alpha_j$ and $\beta_j^\eta$ in $[0,1]$ defining $\alpha_j:=0$ in $[0,1/3]$, $\alpha_j:=s$ in $[2/3,1]$, and $\beta_j^\eta$
as a linear interpolation of the values $1-\eta$ and $1$. Now $\alpha_j$ and $\beta_j^\eta$ are competitors for $\vartheta_p(s)$ and for $j$ large
they satisfy
$$\int_{0}^{1} |1-\beta_j^\eta| \, \sqrt{\psi_p^2(\beta_j^\eta)|\alpha_j'|^2 + |(\beta_j^\eta)'|^2}\, dt\le \vartheta_p(s)+2\eta+\omega(\eta)+\eta\omega(\eta)+\eta^2$$
and this concludes the proof of the claim.

Let us prove now that $\hat{\vartheta}_p(s)\leq\vartheta_p(s)$. We argue exactly as in Proposition~\ref{p2} until estimate 
\eqref{al:p1}. In doing this we point out that $f$, $g$ and $\hat g$ have to be substituted by 
$\psi_p$, $\vartheta_p$ and $\hat\vartheta_p$, respectively.

By keeping the same notation introduced there, we repeat the computations in \eqref{al:com}-\eqref{e:mini}
and we conclude that
\begin{alignat*}1
  \hat \vartheta_p(s)
  &\le \sqrt \eta + 3\eta^2  
+(1+\omega(\eta)) \int_{0}^1
(1-\beta) \sqrt{\psi_p^2(\beta)\,  |\alpha'|^2 + |\beta'|^2}dt \,.
\end{alignat*}
Since the last integral is less than $\vartheta_p(s)+\eta$ and $\eta$ can be made arbitrarily 
small the inequality $\hat{\vartheta}_p\leq\vartheta_p$ follows at once. 
\medskip

We finally prove (iii).
It is easy to check that $f^{(j)}\leq f^{(j+1)}$, and that $f^{(j)}(s)\to \psi_p(s)$ for all $s\in[0,1)$.
Hence, the sequence $(g_j)$ is nondecreasing and $g_j(s)\leq \vartheta_p(s)$ for all $s\geq 0$.
To prove \eqref{e:gjtp}, with fixed $s\in(0,+\infty)$, consider the functionals 
$\mathscr{G}_j,\,\mathscr{G}_\infty:L^1\big((0,1)\big)\times L^1\big((0,1)\big)\to[0,+\infty]$ defined for 
$(\alpha,\beta)\in\mathcal{U}_s$ by
\[
\mathscr{G}_j(\alpha,\beta):=
\int_0^1|1-\beta|\sqrt{(f^{(j)})^2(\beta)|\alpha^\prime|^2+|\beta^\prime|^2}\,dt
\]
and 
\[
\mathscr{G}_\infty(\alpha,\beta):=
\int_0^1|1-\beta|\sqrt{\psi_p^2(\beta)|\alpha^\prime|^2+|\beta^\prime|^2}\,dt
\]
respectively, and set equal to $+\infty$ otherwise on $L^1\big((0,1)\big)\times L^1\big((0,1\big))$.

Note that $\mathscr{G}_j\leq\mathscr{G}_{j+1}$ and that $\mathscr{G}_j$ pointwise converge 
to $\mathscr{G}_\infty$ by Beppo-Levi's theorem. Therefore, $(\mathscr{G}_j)$ $\Gamma$-converges 
to $\overline{\mathscr{G}_\infty}$, the relaxation of $\mathscr{G}_\infty$ w.r.to the $L^1\times L^1$ 
topology.
Being $g_j(s)=\inf\mathscr{G}_j$ and $\vartheta_p(s)=\inf\mathscr{G}_\infty$ to conclude we
need only to discuss the compactness properties of the minimizing sequences of $\mathscr{G}_j$'s.

To this aim let $\alpha_j$, $\beta_j\in W^{1,\infty}\big((0,1)\big)$ be such that 
$\alpha_j(0)=0$, $\alpha_j(1)=s$, $\beta_j(0)=\beta_j(1)=1$, and
\[
\mathscr{G}_j(\alpha_j,\beta_j)\leq g_j(s)+\frac1j.
\]
Hence, either there exists $\delta>0$ and a subsequence $j_k$ such 
that $\inf_{[0,1]}\beta_{j_k}\geq\delta$, or $\inf_{[0,1]}\beta_j\to0$. 
In the former case, we note that
\[
\big(\min_{t\in[\delta,1)}(1-t)f^{(j_k)}(t)\big)\|\alpha_{j_k}^\prime\|_{L^1}\leq g_{j_k}(s)+\frac1{j_k},
\]
and since for $k$ sufficiently large (depending on $\delta$) 
\[
\min_{t\in[\delta,1)}(1-t)f^{(j_k)}(t)=\min_{t\in[\delta,1)}(1-t)\psi_p(t)>0\,,
\]
we conclude that $\sup_k\|\alpha_{j_k}^\prime\|_{L^1}<+\infty$. Moreover, as 
\[
\frac12\sup_j\|\big((1-\beta_j)^2\big)^\prime\|_{L^1}\leq \sup_j g_j(s)\leq\vartheta_p(s),
\]
the sequence $(\alpha_{j_k},\beta_{j_k})$ is pre-compact in $L^1\times L^1$, so that 
by standard properties of $\Gamma$-convergence we may conclude that 
\begin{equation}\label{e:tpcasoa}
\lim_j g_j(s)=\lim_j\inf\mathscr{G}_j=\min\overline{\mathscr{G}_\infty}
=\inf\mathscr{G}_\infty=\vartheta_p(s).
\end{equation}
Instead, in case $\inf_{[0,1]}\beta_j\to0$, set $s_j=\mathrm{argmin}_{[0,1]}\beta_j$ and deduce that
\begin{equation}\label{e:tpcasob}
\lim_jg_j(s)\geq \liminf_j\left(\int_0^{s_j}|1-\beta_j||\beta_j^\prime|\,dt
+\int_{s_j}^1|1-\beta_j||\beta_j^\prime|\,dt\right)\geq1.
\end{equation}
Together with inequality $\vartheta_p(s)\leq 1$, the latter formula provides the conclusion.
\end{proof}

The functionals $F_k^{(k)}$ corresponding to the sequence $(f^{(j)})$ in \eqref{e:fjPhitt} of 
Proposition~\ref{p:tp} provide an approximation of $\Phi_p:L^1(\Omega)\to[0,+\infty]$ defined by
\begin{equation}\label{e:Phitt}
\Phi_p(u):=
  \begin{cases} 
\displaystyle{\int_\Omega |\nabla u|^2dx+\int_{J_{u}}\vartheta_p(|[u]|)d\Hn}
& \textrm{if $u\in GSBV(\Omega)$,}\cr\cr
+\infty & \textrm{otherwise},
  \end{cases} 
\end{equation}
with $\vartheta_p$ is defined in formula \eqref{e:tp}.
\begin{theorem}\label{t:sublin}
Suppose that $(\f j)$ is as in \eqref{e:fjPhitt} above.

Then, the functionals $\Fk kk$ $\Gamma$-converge in $\Lr$ to $\widetilde{\Phi}_p$, where 
\begin{equation}\label{e:p}
\widetilde{\Phi}_p(u,v):=\begin{cases} 
\Phi_p(u) & \textrm{if $v=1$ $\qo$,}\cr
+\infty & \text{otherwise}.
\end{cases} 
\end{equation}
\end{theorem}
\begin{proof}
 By monotonicity of the sequence $(f^{(j)}_k)$ we have that $F_k^{(k)}\geq F_k^{(j)}$ for $k\geq j$, 
 so that by Theorem~\ref{t:gamma-lim} if $\Gamma\hbox{-}\liminf_kF_k^{(k)}(u,v)<+\infty$ then 
 $u\in GBV(\Omega)$, $v=1$ $\Ln$-q.o. on $\Omega$ and for all $j\in\N$
 \[
  \Gamma\hbox{-}\liminf_kF_k^{(k)}(u,1)\geq \Gamma\hbox{-}\lim_kF_k^{(j)}(u,1)
  =\int_\Omega h_j(|\nabla u|)dx+\int_{J_u}g_j(|[u]|)d\Hn+j|D^cu|(\Omega),
 \]
 where $h_j$ and $g_j$ are defined, respectively, by \eqref{h} and \eqref{g} with $f^{(j)}$ in place 
 of $f$. By letting $j\uparrow\infty$, we get that   
\[
h_j(s)\uparrow s^2,\quad\text{ and}\quad g_j(s)\uparrow\vartheta_p(s)\quad \textrm{for all $s\geq 0$.}
\]
Indeed, the former convergence follows from the explicit formula $h_j(s)=s^2$ for $s\in[0,j/2]$ and 
$h_j(s)=js-j^2/4$ for $s\in[j/2,+\infty)$, while the latter in view of (iii) in Proposition~\ref{p:tp}.
Therefore, by Beppo-Levi's theorem we conclude that $u\in GSBV(\Omega)$ with 
 \[
 \Gamma\hbox{-}\liminf_kF_k^{(k)}(u,1)\geq\widetilde{\Phi}_p(u,1).
 \]
 
 To prove the upper bound inequality we note that Lemma~\ref{l:wsub} and \ref{l:boundglimsup} 
 still hold true in this setting as there we have only used that each function $f_k=1\wedge\ve_k^{1/2} f$ 
 in \eqref{fk} is nondecreasing and bounded by $1$ from above, properties enjoyed by $\fk kk$ as well
 (cp. also Theorem~\ref{t:Dugdale}). 
 
 Hence, we may argue again as in Proposition~\ref{p:limsupndim} and reduce ourselves to prove 
 the estimate
 \begin{equation}\label{e:minestp}
 \limsup_{\delta\downarrow 0}\frac 1{\delta^{n-1}}\widehat{F}(u_x,1;x+\delta\, Q_{\nu_u(x)})\leq 
 \vartheta_p(|[u](x)|),
 \end{equation}
 for $u\in SBV^2(\Omega)$ and for $\Hn$-a.e. $x\in J_u$, where $\widehat{F}$ is the 
 $\bar{\Gamma}$-limit of a properly chosen subsequence $(\Fk{k_j}{k_j})$ of $(\Fk kk)$.
 Given \eqref{e:minestp}, we deduce the upper bound estimate as follows: we employ first 
 \cite[Propositions~3.3-3.5]{bou-bra-but} to get the estimate 
 $\widehat{F}(\cdot,1)\leq \widetilde{\Phi}_p(\cdot,1)$
 on the full $SBV$ space, by relaxing the functional $\Phi_\infty:BV(\Omega)\to[0,+\infty]$
\[
\Phi_\infty(u):=\begin{cases}
\displaystyle{
\int_\Omega |\nabla u|^2dx+\int_{J_u}\vartheta_p(|[u](x)|)\,d\Hn} & \textrm{if $u\in SBV^2(\Omega)$}
\cr
+\infty & \textrm{otherwise on $BV(\Omega)$,}
             \end{cases}
\] 
 w.r.to the $w\ast\hbox{-}BV$ topology on $BV(\Omega)$. This implies $\widehat{F}(\cdot,1)\leq \Phi_p$ 
 on $BV(\Omega)$. We get the required estimate on the whole $GSBV\cap L^1(\Omega)$ by the usual 
 truncation argument. We then argue as in Proposition~\ref{p:limsupndim} to show that the whole family 
 $(F_k^{(k)})$ $\Gamma\hbox{-}$converges to $\widetilde{\Phi}_p$.

The proof of 
 \eqref{e:minestp} is identical to the proof of \eqref{e:minest}
in Proposition~\ref{p:limsupndim} and therefore not repeated.
 \end{proof}

\medskip

\subsection{Griffith's brittle fracture}
\label{subsecgriffith}
Finally, we show how to approximate the Mumford-Shah functional by means of any 
sequence $\big(f^{(j)}\big)$ satisfying item (iii) in Proposition~\ref{p:gl}. 
Thus, we recover the original approximation scheme of Ambrosio and Tortorelli 
\cite{amb-tort1}, \cite{amb-tort2} (see also \cite{focardi}).
\begin{theorem}\label{t:MS}
Suppose that $(\f j)$ satisfies 
$\f{j}\leq \f{j+1}$, $\ell_j\uparrow\infty$ and $\f{j}(s)\uparrow \infty$ pointwise in $(0,1).$
Then, the functionals $\Fk kk$ $\Gamma$-converge in $\Lr$ to the functional $\MSt$ 
defined as follows  
\begin{equation}\label{e:Ft}
\MSt(u,v):=\begin{cases} 
\MS(u) & \textrm{if $v=1$ $\qo$,}\cr
+\infty & \text{otherwise}.
\end{cases} 
\end{equation}
\end{theorem}
\begin{proof}[Proof of Theorem~\ref{t:MS}]
We start off as in the proof of Theorem~\ref{t:Dugdale} and note that 
by the very definitions $\Fk jk\leq \Fk kk$ for $j\leq k$ being $(\f j)$ nondecreasing
by assumption. Thus, by Theorem~\ref{t:gamma-lim} we deduce
\begin{equation}\label{e:liminfFj}
\Gamma\hbox{-}\liminf_k \Fk kk(u,v)\geq \F j(u,v),
\end{equation}
where $\F j$ is defined as $F$ in \eqref{F} with $f$ substitued by $\f j$ in formulas
\eqref{h} defining $h_j$, and \eqref{g} defining $g_j$. In particular, the corresponding 
volume density is given by 
\[
h_j(t)=\begin{cases}
t^2 & t\leq\frac{\ell_j}{2} \cr
\ell_j\,t-\frac{\ell_j^2}{4} & t\geq\frac{\ell_j}{2},
       \end{cases}
\]
where $\ell_j$ is the value of the limit in \eqref{f1} and it satisfies $\ell_j\uparrow\infty$. 
Thus $h_j(s)\leq s^2$ and $\lim_jh_j(s)=s^2$ for all $s\in[0,+\infty)$. 
Moreover, the surface energy densities $g_j$ are dominated by the constant $1$, 
and by item (iii) in Proposition~\ref{p:gl} we have $\lim_jg_j(s)=\chi_{(0,+\infty)}(s)$ for all 
$s\in[0,+\infty)$. 
In conclusion, if $\Gamma\hbox{-}\liminf_k\Fk kk(u,v)<+\infty$, by letting $j\uparrow\infty$ in 
\eqref{e:liminfFj} we infer that $v=1$ $\qo$, $u\in GSBV(\Omega)$ and by the Beppo-Levi's
theorem we get
\[
\Gamma\hbox{-}\liminf_k\Fk kk(u,v)\geq \MSt(u).
\]
Eventually, we establish the limsup inequality. Set $\psi:=\chi_{(0,1]}$, we observe once more 
that $\Fk kk\leq AT_k^\psi$ for every $k$, where $AT_k^\psi$ has been defined in \eqref{e:ATk}.
Therefore the conclusion follows by the Ambrosio and Tortorelli result \cite{amb-tort2} (see 
also \cite{focardi}).
\end{proof}
\begin{remark}\label{r:regimes2}
In Remark~\ref{r:regimes} we have shown that both the divergence of the $f_k$'s and the scaling 
with $\ve_k^{1/2}$ in the definition of $f^{(k)}_k$ are influencing the asymptotic behavior of 
the related sequence $(F_k^{(k)})$. 
Here, we show that also the sequence of values of the limits in $1$ of the functions $(1-s)f^{(k)}(s)$ 
is playing a role. 
In particular, we highlight that the pointwise limit of $(f_k^{(k)})$ is not determining 
the asymptotics of $(F^{(k)}_k)$. 

Indeed, suppose that $f^{(k)}(s):=a_k\frac{s}{1-s}$, where $a_k\uparrow\infty$, then 
 \[
 f^{(k)}_k(s)=\begin{cases}
               a_k\ve_k^{1/2}\frac{s}{1-s} & 0\leq s\leq(1+a_k\ve_k^{1/2})^{-1} \cr
               1 & (1+a_k\ve_k^{1/2})^{-1}\leq s\leq 1,
              \end{cases}
 \]
and by letting $k\uparrow\infty$ we infer that 
 \[
 f^{(k)}_k(s)\to \begin{cases}
               \chi_{\{1\}}(s) & \textrm{if $a_k\ve_k^{1/2}\downarrow 0$}   \cr
               \gamma\,\frac{s}{1-s} \wedge 1 & \textrm{if $a_k\ve_k^{1/2}\to\gamma\in(0,+\infty)$}\cr
               \chi_{(0,1]}(s) & \textrm{if $a_k\ve_k^{1/2}\uparrow \infty$}.
              \end{cases}
 \]
Hence, by taking also into account the examples in Remark~\ref{r:regimes}, we have 
built two sequences of functions $f_k^{(k)}$ both converging to $\chi_{\{1\}}$ but giving 
rise in the $\Gamma$-limit on one hand to the Dugdale's cohesive energy and on the other 
hand to a Griffith's type energy.
 \end{remark}

\section*{Acknowledgments} 
Part of this work was conceived when M.~Focardi was visiting the University of Bonn in winter 2014. 
He would like to thank the Institute for Applied Mathematics  for the hospitality and the stimulating scientific 
atmosphere provided during his stay.

M.~Focardi and F.~Iurlano are grateful to Gianni Dal Maso for stimulating discussions and for many 
insightful remarks. They are members of the Gruppo Nazionale per
l'Analisi Matematica, la Probabilit\`a e le loro Applicazioni (GNAMPA)
of the Istituto Nazionale di Alta Matematica (INdAM). 

F.~Iurlano was funded under a postdoctoral fellowship by the Hausdorff Center for Mathematics.
%
%

\end{document}